%% file: main.tex
\documentclass[numbib]{IMAIAI}

\jno{iaaa035} 
\received{30 August 2019} 
\revised{04 August 2020} 
\accepted{05 August 2020} 

\usepackage[ansinew]{inputenc}
\usepackage{color}
\usepackage{amsmath}
\usepackage{amssymb}
\usepackage{amsfonts}
\usepackage{graphicx} 
\usepackage{dsfont}   

\usepackage{mathrsfs} 

\usepackage{array}

\usepackage{mathtools}
\usepackage{amsthm,bm}

\usepackage{algorithm}

\usepackage[table]{xcolor}
\usepackage{multirow}
\usepackage{csquotes}

 \usepackage{pifont}
 \usepackage[noend]{algpseudocode} 
 \usepackage{url}
\newcommand{\R}{\mathbb{R}}

\newcommand{\Rk}{\R^{k}}
\newcommand{\Rdk}{\R^{d\times k}}

\newcommand{\rank}{\operatorname{rank}}
 
\newcommand{\norm}[1]{\lVert#1\rVert_{\mathrm{F}}}
\newcommand{\dotprod}[1]{\langle #1\rangle}

\newcommand{\Tr}[1]{\mathrm{Tr}( #1)}

\newcommand{\KLop}{\mathrm{KL}}
\newcommand{\TVop}{\mathrm{TV}}
\newcommand{\KL}[1]{\KLop(#1)}
\newcommand{\TV}[1]{\TVop(#1)}
\newcommand{\ro}[1]{\rho(#1)}
\newcommand{\normop}[1]{\lVert#1\rVert_{\mathrm{op}}}
\newcommand{\vect}[1]{\mathrm{vec}(#1)}

\newcommand{\Od}{\mathcal{O}(d)}
\newcommand{\Gkd}{\mathrm{G}(k,d)}

\newcommand{\TODO}[1]{[#1]}

\graphicspath {{figures/}}

\begin{document}

\title{The generalized orthogonal Procrustes problem\\in the high noise regime}

\shorttitle{Generalized orthogonal Procrustes in high noise} 
\shortauthorlist{Thomas Pumir, Amit Singer, Nicolas Boumal} 

\author{{
\sc Thomas Pumir}$^*$,\\
Department of Operations Research and Financial Engineering, Princeton University\\
$^*${\email{Corresponding author. Email: pumir.thomas@gmail.com}}\\
{\sc Amit Singer}\\
Mathematics Department and PACM, Princeton University\\
{amits@math.princeton.edu}\\
{\sc and}\\
{\sc Nicolas Boumal} \\
Mathematics Department, Princeton University\\
{nboumal@math.princeton.edu}}

\maketitle

\begin{abstract}
{We consider the problem of estimating a cloud of points from numerous noisy observations of that cloud after unknown rotations, and possibly reflections.
This is an instance of the general problem of estimation under group action, originally inspired by applications in
3-D imaging and computer vision.
We focus on a regime where the noise level is larger than the magnitude of the signal, so much so that the rotations cannot be estimated reliably.
We propose a simple and efficient procedure based on invariant polynomials (effectively: the Gram matrices) to recover the signal, and we assess it against fundamental limits of the problem that we derive. We show our approach adapts to the noise level and is statistically optimal (up to constants) for both the low and high noise regimes.
In studying the variance of our estimator, we encounter the question of the sensitivity of a type of thin Cholesky factorization, for which we provide an improved bound which may be of independent interest.}
Procrustes; Cholesky factorization.
\\
2000 Math Subject Classification: 34K30, 35K57, 35Q80,  92D25
\end{abstract}

\input{Introduction}
\input{EstimationAlgorithm}
\input{StabilityEstimator}

\input{LowerBound}

\input{NoiseRegimes}

\input{Numerics}

\input{Conclusion}

\bibliographystyle{IMAIAI}
\bibliography{pumir}

\clearpage

\appendix

\input{Distance}
\input{ImpossibilityEstimation}

\input{Lemmas}
\input{Cholesky}

\input{MinimaxProofs}
\input{KLBound}
\input{MSE}

\end{document}

%% file: Introduction.tex
\section{Introduction}

We consider the problem of estimating  $k$ labeled points in $\mathbb{R}^{d}$, with $k \geq d$.
This cloud of points, which we call the parameter, is represented as a matrix $X$ of size $d \times k$. We restrict ourselves to the case where the smallest singular value of $X$ is bounded away from zero, that is, the cloud spans all $d$ dimensions.
We observe $N$ independent measurements $Y_{1},\ldots,Y_{N}$ of $X$, following the model
\begin{equation}\label{eq::description}
	Y_{i} = Q_{i}X + \sigma E_{i},
\end{equation}
where $\sigma > 0$ is the standard deviation of the noise, $E_1, \ldots, E_N$ are independent noise matrices in $\mathbb{R}^{d \times k}$ with independent, standard Gaussian entries and $Q_1, \ldots, Q_N$ are drawn uniformly and independently at random from the orthogonal group,
\begin{align}
	\Od & = \{Q \in \R^{d \times d}: Q^{T}Q = I_{d}\},
	\label{eq:Od}
\end{align}
where $I_d$ is the identity matrix of size $d$.
In what follows, we refer to orthogonal matrices as rotations, bearing in mind that (in our meaning) they may also include a reflection.
%
The method we propose in Section~\ref{sec:estimationalgorithm} applies under relaxed assumptions on the distributions of $Q_i$ and $E_i$: our assumptions here serve to streamline exposition.
This problem we investigate belongs to a larger class of estimation problems under group actions~\citep{BandeiraRigolletWeed2018RatesMRA}.

Notice that the distribution of the observations $Y_i$ is unchanged if $X$ is replaced by $QX$, for any orthogonal $Q$. As a result, we can only hope to recover the cloud $X$ up to a global rotation.
Accordingly, we define an equivalence relation $\sim$ over $\mathbb{R}^{d \times k}$:
\begin{equation}\label{eq::equivalencerelation}
X_{1} \sim X_{2} \iff X_{1} = QX_{2} \text{ for some } Q \in \Od.
\end{equation}
This equivalence relation partitions the parameter space into equivalence classes
\begin{align*}
	[X] & = \{ QX: Q \in \Od\}.
\end{align*}
The set of equivalence classes is the quotient space $\mathbb{R}^{d \times k}/\! \sim$. The distribution of the measurements $Y_i$ is parameterized by $[X]$, which we aim to estimate.

A natural approach to estimate $[X]$ would be to estimate the rotations $Q_{i}$ first (seen as latent or nuisance variables),  align the observations $Y_i$ using the estimated rotations, and average.
Typical of those approaches, the maximum likelihood estimator (MLE) is a solution of the following non-convex optimization problem:
\begin{equation}\label{eq::MLE}
\underset{\hat{X} \in \R^{d \times k},\hat{Q}_{1},\ldots,\hat{Q}_{N} \in \Od}{\mathrm{min}} \ \sum_{i=1}^{N}\norm{\hat{Q}_{i}^TY_{i}^{} - \hat{X}}^{2}.
\end{equation}
For any fixed choice of estimators $\hat{Q}_1, \ldots, \hat{Q}_N$, the optimal estimator for $\hat{X}$ according to the above is
\begin{equation}\label{MLEestimate}
\hat{X} = \dfrac{1}{N} \sum_{i=1}^{N}\hat{Q}_{i}^{T}Y_{i}^{}.
\end{equation}
This estimator can be plugged into the cost function of \eqref{eq::MLE}, reducing the problem to that of estimating only the rotations, upon which $\hat X$ can be deduced from~\eqref{MLEestimate}. A number of papers focus on the resulting problem, called synchronization of rotations~\citep{singer2010angular}.

Such approaches, however, necessarily fail at low signal-to-noise ratio (SNR). Specifically, we argue that if the noise level is too large,
no procedure can reliably determine the latent rotations $Q_i$. Essentially, this is because, when $\sigma$ is too large, the distribution of $QX + \sigma E$ is indistinguishable from that of $Q'X + \sigma E$, where $Q$ and $Q'$ are two rotations.

To see this, consider the following strictly simpler problem: we observe $Y = sQX + \sigma E$, where $Q \in \mathcal{O}(d)$, 
$X \in \R^{d \times k}$ and $\sigma > 0$ are \emph{known}, while $E \in \R^{d\times k}$ has independent standard Gaussian entries and $s$ is uniformly sampled from $\{+1, -1\}$, both unknown. An estimator $\psi$ (deterministic) assigns an estimate of $s$ to an observed $Y$. Estimating $s$ in this context is strictly simpler than estimating $Q_i$'s to any reasonable accuracy in our model (where furthermore $X$ is unknown), as we are only asked to determine whether $Q_i$ is some given rotation, or its opposite. Yet, even this simpler problem is hopeless at low SNR, directly implying the impossibility of estimating the rotations in problem~\eqref{eq::MLE}:
\begin{proposition}\label{prop::ImpossibilityRotations}
For any tolerance $\tau \in (0, 1/2)$, there exists a critical noise level $\sigma_0$ such that, for any estimator $\psi$, if $\sigma > \sigma_0$, then the probability of error $\mathbb{P}(\psi(Y) \neq s)$ exceeds $\tau$. 
\end{proposition} \noindent (The proof relies on the optimality of the likelihood ratio test for Gaussian distributions, see Appendix~\ref{ImpossibilityEstimation}.)

Because of this fundamental obstruction, in this paper, we aim to estimate $[X]$ directly from observations $Y_i$, bypassing any estimation (even implicit) of the latent $Q_i$'s. To do so, we follow a trend in signal processing that consists in estimating $[X]$ from features of the observations that are \emph{invariant} under the group action of $\Od$
~\citep{Tukey84Spectral},~\citep{SadlerGiannakis92reconstruction},~\citep{Giannakis89Reconstruction},~\citep{Abbeetal18AperiodicMRA},~\citep{BandeiraRigolletWeed2018RatesMRA},~\citep{Perryetal2017SampleComplexityMRA},~\citep{boumal2017heterogeneous}.

Specifically, consider the Gram matrix of $X$, that is, $X^T X$: it is invariant under orthogonal transformations since $(QX)^T(QX) = X^TX$ for any $Q \in \Od$. Thus, up to noise terms that we will handle, the Gram matrices of the observations $Y_i$ reveal information about the Gram matrix of $X$ without the need to estimate the $Q_i$'s. In Section~\ref{invariant}, we call upon invariant theory to argue that no other polynomial invariant features are necessary, in the sense that (a) they would be redundant with the Gram matrix, and (b) the Gram matrix is sufficient to fully characterize the equivalence class $[X]$.

Based on these observations, we proceed (still in Section~\ref{invariant}) to derive an estimator for the Gram matrix of $X$ from the given observations~\eqref{eq::description}---this also requires estimating the noise level $\sigma$, which we discuss. From the estimated Gram matrix, we construct an estimator for the equivalence class $[X]$.
Since we only have access to a finite number $N$ of observations, we can only hope to recover an approximation of the Gram matrix. 
Accordingly, in Section~\ref{sec:stability} we study the sensitivity of the mapping from the Gram matrix to the sought equivalence class $[X]$. This reduces to showing stability of the factorization of a rank-$d$ positive semidefinite matrix. To address this question, we propose a new analysis of such matrix factorization, with a geometric proof.

In Section~\ref{sec:statisticaloptimality}, we show that the proposed approach is statistically optimal, that is, it makes the best use of available samples.
Moreover, we show that the mean squared error (MSE) of our estimator behaves like $O\left(\sigma^2/N + \sigma^4/N\right)$, which is shown to be  adaptively optimal.
Indeed, this highlights the existence of two regimes: at low SNR, our estimator's MSE behaves like $O(\sigma^{4}/N)$,
while at high SNR we obtain an MSE of order $O(\sigma^{2}/N)$ which matches the regime we would get if the rotations were known. In particular, the MSE can be driven to zero at \emph{any} noise level, provided the number of observations $N$ is sufficiently large. We give an explicit characterization of those two regimes and we support our claim with numerical simulations in Section~\ref{sec:XP}. 

To make sense of MSE in estimating $[X]$, we need a distance on the quotient space. A natural choice consists in computing the Frobenius distance between two aligned representatives, that is,
\begin{equation}\label{eq::distance}
	\ro{[X_{1}],[X_{2}]} = \min_{Q \in \Od}\norm{X_{1} - QX_{2}},
\end{equation}
where $\norm{X} = \sqrt{\Tr{X^T X}}$. We sometimes write simply $\rho(X_1, X_2)$, where it is clear that we mean the distance between $[X_1]$ and $[X_2]$.
Computing this distance is known as the orthogonal Procrustes problem: it can be done efficiently via singular value decomposition (SVD)~\citep{Schonemann66Procrustes} (see also Appendix~\ref{distance}). 

\subsection*{Related work}


Procrustes problems consist in finding correspondences between shapes that have been transformed through translation, rotation or dilation~\citep{Schonemann66Procrustes,tenberge1977procrustes}. They notably find application in multivariate analysis~\citep{Procrustes62,AppliedMultivariate76}, multidimensional scaling~\citep{MDS05}, computer vision~\citep{Zhang00CameraCalibration} and natural language processing~\citep{XingWangLiuLin15OrthogonalEmbeddings,Smithetal17OfflineBilingualWordvectors,GraveJoulinBerthet18WassersteinProcrustes}. 
An extensive survey on applications of the Procrustes problem can be found in~\citep{Procrustes05}.

A particular line of work on such problems has focused on estimating each rotation before estimating the orbit (that is, equivalence class) of the matrix $X$.
In particular, the  MLE~\eqref{eq::MLE} is cast into the following non-convex optimization problem:
\begin{equation}\label{eq::Gro1}
\min_{\hat{Q}_{1},\ldots,\hat{Q}_{N} \in \Od} \sum_{i \neq j}\norm{\hat{Q}_{i}Y_{i} - \hat{Q}_{j}Y_{j}}^{2},
\end{equation}
or equivalently:
\begin{equation}\label{eq::Gro2}
\max_{\hat{Q}_{1},\ldots,\hat{Q}_{N} \in \Od} \sum_{i \neq j}\dotprod{\hat{Q}_{i}Y_{i},\hat{Q}_{j}Y_{j}},
\end{equation}
where $\dotprod{A,B} = \Tr{A^T B}$ is the inner product we use throughout.
Among many, this problem was investigated by Nemirovski~\citep{nemirovski2007quadratic}
and Man-Cho So~\citep{Mansocho2010momentinequalities}.
A particularization of this problem amounts to the little Grothendieck problem~\citep{khot2012grothendieck,naor2013efficient}.
A natural way to deal with the non-convexity and ensuing computational complexity
is to study a convex relaxation of~\eqref{eq::Gro2}.
Following the line initiated by the seminal work of Goemans \& Williamson~\citep{goemans1995maxcut}, Bandeira et al. \citep{bandeira2016Grothendieck}
and many others investigated such semidefinite relaxations.
In contrast, we focus on a noise regime where $Q_i$ s cannot be estimated; hence these approaches cannot succeed.

The problem we consider finds its original motivation in cryo-electron microscopy (cryo-EM), an imaging technique used in structural biology to estimate the three-dimensional shape of a molecule from two-dimensional projections of that molecule under random and unknown orientations---see~\citep{Singer2018MathematicsFC} for a review of mathematical aspects of this task.
Our setting and approach are also closely connected to the multi-reference alignment problem and recent literature on the topic, where noisy realizations of a randomly, cyclically
shifted version of a vector are observed~\citep{Bandeiraetal14MRASDP}.

Close to the approach taken in the present work, Giannakis~\citep{Giannakis89Reconstruction}, Sadler and Giannakis~\cite{SadlerGiannakis92reconstruction} 
and more recently Bendory et al.~\citep{bendory2018bispectruminversion} and Abbe et al.~\citep{Abbeetal18AperiodicMRA} among others,
consider estimating moments of the signal that are invariant under cyclic shifts.
The advantage of such formulation is its validity for any SNR regime, provided sufficiently many samples are available.
In particular, this approach has been shown to be optimal both in terms of rate of estimation~\citep{BandeiraRigolletWeed2018RatesMRA} and sample complexity~\citep{Perryetal2017SampleComplexityMRA}. 

%% file: EstimationAlgorithm.tex
\section{Invariant features approach}\label{invariant}


We focus on estimating $[X]$ from samples drawn according to~\eqref{eq::description}, disregarding the variables $Q_1, \ldots, Q_N$. Based on the discussion above, it is apparent that any method for estimating $[X]$ which (explicitly or implicitly) relies on estimating the rotations reliably must fail beyond a certain noise level. Hence, 
we take a different approach.
For each observation $Y_i$, we compute polynomial functions of $Y_i$ that, aside from the noise $E_i$, are invariant under the rotation $Q_i$. Such functions are called \emph{invariant features}. Then, our estimation problem reduces to that of estimating $[X]$ from the estimated invariant features. This approach completely bypasses estimation of the latent variables. 
We first explain the method with some background below, before showing that it does not break down at high noise levels.

\subsection{Invariant polynomials}

Invariant theory is concerned with polynomials that are invariant under some group action. In our case, this specializes to the following central definition.
\begin{definition}
A multivariate polynomial $p \colon \mathbb{R}^{d \times k} \to \mathbb{R}$ is said to be invariant under the action of the orthogonal group if, for all $X \in \mathbb{R}^{d \times k}$ and for all $Q \in \Od$,
\begin{align*}
p(QX) & = p(X).
\end{align*}
\end{definition}
The goal of this section is to identify all invariant polynomials in our specific setting. A standard observation in invariant theory is that it is sufficient to consider \emph{homogeneous} invariant polynomials.
\begin{proposition}
Consider a multivariate polynomial $p$ of degree $r$, decomposed into a sum of homogeneous parts $p_1, \ldots, p_r$ where $p_i$ has degree $i$:
\begin{align*}
	p(X) & = p_{0} + p_{1}(X) + \cdots + p_{r}(X).
\end{align*}
If $p$ is invariant under the action of the orthogonal group, that is, if for all $Q \in \Od$, $p(QX) \equiv p(X)$,
then each homogeneous part $p_i$ is itself invariant under that action.
\end{proposition}
\begin{proof}
By invariance of $p$, we obtain for all $X$ and orthogonal $Q$ that
\begin{align*}
p_{0} + p_{1}(X) + \cdots + p_{r}(X) = p(X) = p(QX) = p_{0} + p_{1}(QX) + \cdots + p_{r}(QX).
\end{align*}
By identifying the homogeneous parts of the polynomials, we obtain for all $i$ that $p_{i}(X) \equiv p_{i}(QX)$ for all $Q \in \Od$, hence each $p_i$ is itself invariant.
\end{proof}

Homogeneous invariant polynomials necessarily have even degree. In particular, there are no interesting invariant polynomials of degree zero or one.
\begin{proposition}
No homogeneous polynomial of odd degree is invariant under orthogonal group action.
\end{proposition}
\begin{proof}
By contradiction, let $p(X)$ be a homogeneous polynomial of odd degree, invariant under orthogonal group action. Then, $p(X) = p(QX)$ for all orthogonal $Q$. This holds in particular for $Q = -I_d$, so that $p(X) = p(-X) = -p(X)$ (we used that the degree is odd in the last equality). Thus, $p(X) = 0$ for all $X$, which contradicts the fact that $p$ has odd degree.
\end{proof}

We now give an elementary statement of a key property of degree-two invariant polynomials: they can be expressed as a linear function of the \textit{Gram matrix} of $X$, namely $X^{T}X$. We use the notation $[k] = \{1, \ldots, k\}$ and $\langle A, B \rangle = \mathrm{Tr}(A^T B)$.
\begin{proposition}
Any 
homogeneous polynomial $p(X)$  of degree two that is invariant under the action of the orthogonal group is a linear combination of scalar products between the column vectors of $X$, denoted by $x_{1},\ldots,x_{k}$.
In other words, there exist coefficients $m_{ij}$ for all $i,j \in [k]$ such that
\begin{equation}\label{eq::invPoly}
p(X) = \sum_{i,j}m_{ij} \langle x_{i}, x_{j} \rangle = \langle M, X^T X \rangle.
\end{equation}
\end{proposition}

\begin{proof}
Any homogeneous polynomial of degree two in $X$ can be written as
\begin{align*}
p(X) &= \langle \vect{X}\vect{X}^{T}, A \rangle, 
\end{align*}
where $\vect{X}$ vectorizes a matrix by stacking its columns, and $A$ is a coefficient matrix of size $dk \times dk$. Using the property $\vect{ABC} = (C^T \otimes A)\vect{B}$ where $\otimes$ denotes the Kronecker product, we find that
\begin{align*}
p(QX) &= \langle (I_{k} \otimes Q)\vect{X}\vect{X}^{T} (I_{k} \otimes Q)^{T}, A \rangle \\
&= \langle \vect{X}\vect{X}^{T} , (I_{k} \otimes Q)^{T}A (I_{k} \otimes Q) \rangle. 
\end{align*}
If $p$ is invariant, then $p(X) = p(QX)$ for all $X$ and $Q \in \Od$. Hence, by identification:
\begin{align*}
A &= (I_{k} \otimes Q)^{T}A(I_{k} \otimes Q), \quad \text{ for all } Q \in \Od.
\end{align*}
Let $A$ be a block matrix with blocks $A_{ij} \in \R^{d \times d}$ for $i,j$ ranging in 
$[k]$. The above states:
\begin{align*}
A_{ij} &= Q^{T}A_{ij}Q,
\end{align*}
for all $i,j \in [k]$ and $Q \in \Od$. This has the following consequences:
\begin{enumerate}
\item Considering each $Q$ in the set of diagonal matrices with diagonal entries in $\{-1,+1\}$ shows that 
all the elements off the diagonal of $A_{ij}$ are equal to their opposite, implying that $A_{ij}$ is diagonal.
\item Considering each $Q$ in the set of permutation matrices implies that all the diagonal elements are equal.
\end{enumerate}
Thus, $A_{ij} = m_{ij}I_{d}$ and $A = M \otimes I_{d}$, with $M \in \R^{k \times k}$.
As a result:
\begin{align*}
p(X) &= \dotprod{\vect{X},(M \otimes I_{d})\vect{X}} \\
&= \dotprod{\vect{X},\vect{XM^{T}}} \\
&= \dotprod{X,XM^{T}} \\
&= \dotprod{X^{T}X,M}. 
\end{align*}
In words, $p(X)$ is a linear combination of the Gram matrix entries.
\end{proof}

A question that naturally arises is: what are the other invariant polynomials?
The first fundamental theorem of the orthogonal group (see~\citep[Theorem.~14-1.2]{KacNotes} for instance) provides an answer to this question.
\begin{theorem}\label{thm::fundamentalThm}
The functions $g_{ij}\colon(x_{1},\ldots,x_{k}) \mapsto \dotprod{x_{i},x_{j}}, 1 \leq i \leq j \leq k$,
generate the ring of invariant polynomials, that is: any invariant polynomial is a polynomial combination of the degree two invariants $g_{ij}$.
\end{theorem}
Thus, invariant polynomials of degree higher than two do not carry further information about $[X]$.
The next natural question is: are the invariants sufficient to fully characterize the equivalence classes?
The following classical theorem from invariant theory provides a positive answer to this question in the case of compact groups (see~\citep[Theorem.~6-2.2]{KacNotes} for instance).
\begin{theorem}\label{thm::carac}(Informal)
The full invariant ring characterizes the equivalence classes.
\end{theorem}
Theorems~\ref{thm::fundamentalThm} and~\ref{thm::carac} combined imply that $X_1^TX_1^{} = X_2^TX_2^{}$ if and only if $[X_1] = [X_2]$, which is a well-known fact here derived through the prism of invariant features.

\subsection{Estimation algorithm} \label{sec:estimationalgorithm}

%

Above, we have shown that the problem of recovering $[X]$ can be reduced (without loss) to that of estimating the Gram matrix $G = X^TX$, with the advantage that the latter is invariant under orthogonal transformations.
We build on this observation to propose a concrete algorithm.
Consider the Gram matrix of an observation as in equation.~\eqref{eq::description}:
\begin{equation}
	Y_{i}^T Y_{i}^{} = X^T X + \sigma\left( X^{T}Q_{i}^{T}E_{i}^{} + E_{i}^{T}Q_{i}^{}X\right) + \sigma^2 E_{i}^T E_{i}^{}.
\end{equation}
By the strong law of large numbers, their empirical mean converges almost surely to their expectation (a characterization of the fluctuations for finite $N$ follows):
\begin{align}
	\hat M_N & = \frac{1}{N}\sum_{i=1}^{N}Y_{i}^{T}Y_{i}^{} \underset{N\to\infty}{\longrightarrow} X^{T}X + d\sigma^{2}I_{k}.
	\label{eq:hatMN}
\end{align}
Here, we used independence of the $Y_i$'s, independence of $Q_i$ and $E_i$ for each $i$, and the fact that individual entries of each $E_i$ are independent with mean zero and unit variance.

If the noise level $\sigma$ is known, we can get an unbiased estimator for $X^T X$ as
\begin{align}
	\hat G_N & = \hat M_N - d\sigma^2 I_k.
	\label{eq:hatGN}
\end{align}
Since $\hat G_N$ is expected to be close to $X^TX$ for large $N$, it is reasonable to consider an estimator $[\tilde X_N]$ for the equivalence class $[X]$ where $\tilde X_N$ is a solution of the optimization problem
\begin{align}
	\min_{\hat{X} \in \mathbb{R}^{d \times k}} \norm{\hat{G}_{N} - \hat{X}^{T}\hat{X}}.
	\label{factor}
\end{align}
Well-known extremal properties of the eigenvalue decomposition of a symmetric matrix tell us that an optimizer $\tilde X_N$ can be obtained by computing $d$ dominant, orthonormal eigenvectors of $\hat G_N$, and scaling them by the square root of their corresponding eigenvalues. Since $\hat G_N$ and $\hat M_N$ share the same eigenvectors, with eigenvalues related through $\lambda_\ell(\hat G_N) = \lambda_\ell(\hat M_N) - d\sigma^2$, this computation can equivalently be executed from $\hat M_N$ directly. This procedure is summarized in Algorithm~\ref{alg:estimation}. Provided $\hat M_N$ has an eigengap separating its $d$th and $(d+1)$st largest eigenvalues, this procedure uniquely defines $[\tilde X_N]$. In Section~\ref{sec:stability}, we argue that such an eigengap exists with high probability if $N$ is sufficiently large, and we bound the error $\rho([X], [\tilde X_N])$.


We note that the Gram matrix estimator $\hat G_N$ could be replaced by a more robust estimator, for example based on median-of-means as favored in~\citep{Bandeira2018EstimationGroupAction}, leveraging work by~\TODO{Nemirovsky and Yudin} \citep{NemirovskyYudin83OptimizationComplexity} and more recently by~\TODO{Joly et al.} \citep{joly2017meanrandomvector}. Such refinements are not necessary under our assumption of Gaussian noise but could be useful for heavy-tailed noise.

\begin{algorithm}[t]
\caption{Estimation algorithm with $\sigma$ given}\label{alg:estimation}
\begin{algorithmic}[1]
	\State Compute the sample mean of the Gram matrices: $\hat{M}_N = \dfrac{1}{N}\displaystyle\sum_{i=1}^{N}Y_{i}^{T}Y_{i}^{}$.
	\State Compute $d$ top eigenvalues $\lambda_1, \ldots, \lambda_d$ of $\hat M_N$ with associated orthonormal eigenvectors $v_1, \ldots, v_d \in \Rk$.
	\State Define the scaling factors $\alpha_i = \sqrt{\max(0, \lambda_i - d\sigma^2)}$ for $i = 1, \ldots, d$.
	\State Form $\tilde X_N \in \Rdk$ with rows $\alpha_i^{} v_i^T$ for $i = 1, \ldots, d$.
Our estimator is $[\tilde X_N]$.
\end{algorithmic}
\end{algorithm}


\subsection{Estimation when $\sigma$ is unknown}

If $\sigma$ is unknown, it can be estimated from the eigenvalues of $\hat M_N$~\eqref{eq:hatMN}. Indeed, in the limit of $N$ going to infinity, the $k-d$ smallest eigenvalues are equal to $d\sigma^2$. For finite $N$, they fluctuate around that value. Thus, with $\lambda_1 \geq \cdots \geq \lambda_k$ 
the eigenvalues of $\hat M_N$, a possible estimator $\hat \sigma_N \geq 0$ for $\sigma$ is defined by
\begin{align}\label{eq::SigmaEstimator}
	d \hat\sigma_N^2 & = \frac{1}{k-d} \sum_{\ell = d+1}^{k} \lambda_\ell = \frac{1}{k-d} \left( \Tr{\hat M_N} - (\lambda_1 + \cdots + \lambda_d) \right),
\end{align}
where the second form is computationally favorable.
The resulting procedure is summarized as Algorithm~\ref{alg:estimationnosigma}.

We mention that in the case where $X$ is centered, that is, when $X\bm{1} = 0$, $\sigma$ can be reliably estimated 
by computing the empirical variance of the $dN$  i.i.d.\ entries of samples $\frac{1}{\sqrt{k}} Y_i \bm{1} \sim \mathcal{N}(0, \sigma^2 I_d)$ for $i = 1,\ldots,N$.


\begin{algorithm}[t]
	\caption{Estimation algorithm with $\sigma$ unknown}\label{alg:estimationnosigma}
	\begin{algorithmic}[1]
		\State Execute steps 1 and 2 of Algorithm~\ref{alg:estimation}.
		\State Define $\hat\sigma_N \geq 0$ such that $\hat\sigma_N^2 = \dfrac{1}{d(k-d)}\Big( \Tr{\hat{M}_N} - (\lambda_1 + \cdots + \lambda_d) \Big)$. 
		\State Define the scaling factors $\alpha_i = \sqrt{\max(0, \lambda_i - d\hat \sigma_N^2)}$ for $i = 1, \ldots, d$.
		\State Execute step 4 of Algorithm~\ref{alg:estimation}.
	\end{algorithmic}
\end{algorithm}


%% file: StabilityEstimator.tex
\section{Stability of the estimator} \label{sec:stability}

The consistency of $\hat G_N$~\eqref{eq:hatGN} as an estimator of the Gram matrix is guaranteed by the law of large numbers.
However, for finite $N$, we can only hope to estimate the Gram matrix approximately: we characterize the expected errors here. From this approximate Gram matrix, we obtain our estimator $[\tilde X_N]$ by solving the optimization problem~\eqref{factor} (which is akin to forming a type of thin Cholesky factorization of $\hat G_N$ after projecting the latter to the positive semidefinite matrices): we call this step the \emph{Gram inversion}. To understand the final error on our estimator, we need to study the sensitivity of Gram inversion: we start with this.

\subsection{Sensitivity of Gram inversion}


Consider the function $f([X]) = X^T X$. This is a map from the quotient space 
\begin{align}
	\mathcal{M} = \R^{d \times k}_* / \! \sim,
	\label{eq:calM}
\end{align}
where $\sim$ is the equivalence relation defined in~\eqref{eq::equivalencerelation} and $\R^{d\times k}_*$ is the set of matrices in $\R^{d\times k}$ of full rank $d$, to the set
\begin{align}
	\mathcal{N} = \{ G \in \mathrm{Sym}_{k} : \rank(G) = d, G \succeq 0 \},
	\label{eq:calN}
\end{align}
where $\mathrm{Sym}_{k}$ is the set of symmetric matrices of size $k$ and $G \succeq 0$ means $G$ is positive semidefinite.
In Section~\ref{invariant},  we argued that $[X]$ can be recovered uniquely from $X^T X$, meaning that
$f$ is globally invertible. 
In this section, we are concerned with the sensitivity of the inverse, $f^{-1}$.

A bound on the sensitivity appears in~\citep[Lem.~5.4]{Tuetal2016LowRankProcrustesFlow}. We repeat it here.
\begin{lemma} \label{lem:TuLemma}
	For any $[X], [\tilde X] \in \mathcal{M}$,
	\begin{align*}
		\rho([X], [\tilde X]) & \leq \frac{L}{\sigma_d(X)} \norm{X^T X - \tilde X^T \tilde X}, & \textrm{ with } & & L & = \frac{1}{\sqrt{2(\sqrt{2}-1)}}.
	\end{align*}
\end{lemma}
\noindent This result establishes a Lipschitz constant for $f^{-1}$ in the vicinity of $[X]$, with respect to the distance $\rho$ on $\mathcal{M}$~\eqref{eq::distance} and the Frobenius distance on $\mathcal{N}$.
Through a geometric argument, we confirm that the coefficient $\sigma_d(X)$ in the denominator cannot be avoided, and we show $L$ must be at least $1/\sqrt{2}$ (that is $0.71..$ compared to $1.10..$ above). Then, we lean on Lemma~\ref{lem:TuLemma} to obtain a new bound with (essentially) that optimal constant. To do so, we use the proof mechanics proposed by~\TODO{Chang and Stehle} \citep{ChangStehle10RigorousBoundsMatrixFactorization} in their study of the stability of the Cholesky decomposition of (strictly) positive definite matrices.

We equip the quotient space $\mathcal{M}$ with a smooth structure as a Riemannian quotient manifold of $\Rdk_*$ with the standard inner product $\dotprod{\cdot, \cdot}$~\citep[\S3.4]{AMS08}.\footnote{In contrast, the quotient space $\Rdk/\!\sim$ (without rank restriction) does not admit such a smooth structure, because not all its equivalence classes have the same dimension as submanifolds of $\Rdk$.} Likewise, we endow $\mathcal{N}$ with a smooth structure as a Riemannian submanifold of $\mathrm{Sym}_k$ with the standard inner product $\dotprod{\cdot, \cdot}$, as in~\citep{vandereycken2009psdfixedrank}. With these smooth structures, $f \colon \mathcal{M} \to \mathcal{N}$ is a smooth function. Detailed background on both geometries and their relations can be found in~\citep{MassartAbsil18QuotientPSD}.

The distance $\rho$ happens to be the geodesic distance on $\mathcal{M}$~\citep{MassartAbsil18QuotientPSD}. Furthermore, since $\mathcal{N}$ is a Riemannian submanifold of $\mathrm{Sym}_{k}$ equipped with the trace inner product, the Frobenius distance between close-by points of $\mathcal{N}$ is an excellent approximation for the geodesic distance between them. As a result, locally around $X^T X$, the operator norm (the largest singular value) of the differential of $f^{-1}$ at $X^TX$ reveals the local Lipschitz constant of $f^{-1}$ with respect to these distances.

The inverse function theorem~\citep[Thm.~4.5]{lee2012smoothmanifolds} states that the differential of $f^{-1}$ at $X^TX$ is the inverse of the differential of $f$ at $[X]$. Accordingly, we first study the singular values of the differential of $f$ at $[X]$.
%
%

As a preliminary step, for $X \in \R^{d \times k}_*$, 
consider the differential $\mathcal{L}_X$ of the map $X \mapsto X^T X$ (its relation to $f$ is elucidated below):
\begin{align}\label{eq::operator}
\begin{array}{ccccc}
\mathcal{L}_{X} & : &  \R^{d \times k}  & \to & \mathrm{Sym}_{k} \\
 & & \dot X & \mapsto &  X^T \dot X + \dot X^T X.
\end{array}
\end{align}
%
Clearly, the following subspace is included in the kernel of $\mathcal{L}_X$:
\begin{align}
	\mathrm{V}_X & = \{ \Omega X : \Omega + \Omega^T = 0 \}.
\end{align}
We can thus restrict our attention to the orthogonal complement of $\mathrm{V}_X$, which we denote by $\mathrm{H}_X = (\mathrm{V}_X)^\perp$. By definition, $\dot X \in \R^{d \times k}$ is orthogonal to $\mathrm{V}_X$ if and only if $\dotprod{\dot X,\Omega X} = 0$ for all skew-symmetric matrices $\Omega$, hence:
\begin{align}\label{eq::horizontal}
	\mathrm{H}_X & = \{ \dot X \in \R^{d \times k} : \dot X X^T = X \dot X^T \}.
\end{align}
Restricted to $\mathrm{H}_X$, the smallest singular value of $\mathcal{L}_X$ is positive.
\begin{proposition}\label{lem::operatorInverse}
	Given $X \in \R^{d \times k}_*$, with notation as above, consider the restriction of $\mathcal{L}_X$ as an operator from $\mathrm{H}_X$ to $\mathcal{L}_X(\mathrm{H}_X)$. That operator is invertible and its smallest singular value is $\sqrt{2}\sigma_d(X)$.
\end{proposition}
\begin{proof}
See Appendix \ref{CholAppendix}. 
\end{proof}

Using tools from differential geometry~\citep[Section 3.5.8]{AMS08}\citep{MassartAbsil18QuotientPSD}, it can be shown that $\mathrm{H}_X$ (equipped with the standard inner product) is isometric to the tangent space of $\mathcal{M}$ at $[X]$, so that the singular values of $\mathcal{L}_X$ on $\mathrm{H}_X$ are equal to the singular values of the differential of $f$ at $[X]$. Thus, calling upon the inverse function theorem, we conclude that the largest singular value of the differential of $f^{-1}$ at $X^T X$ is $1/\sqrt{2}\sigma_d(X)$. In turn, this shows the constant $L$ in Lemma~\ref{lem:TuLemma} must be at least $\frac{1}{\sqrt{2}}$, and the coefficient $\sigma_d(X)$ cannot be removed.

Lemma~\ref{lem:TuLemma} and Proposition~\ref{lem::operatorInverse} together allow us to show our main result regarding the local sensitivity of Gram inversion, using a technique by Chang \& Stehle~\citep{ChangStehle10RigorousBoundsMatrixFactorization}.

\begin{theorem}\label{thm::cholesky}

Consider two matrices $X, \tilde{X} \in \R^{d\times k}_*$. If their Gram matrices $G = X^T X$ and $\tilde{G} = \tilde{X}^T\tilde{X}$ are close, specifically, if
 \begin{align*}
 \norm{G - \tilde{G}} &\leq \frac{\sigma_{d}^{2}(X)}{2},
 \end{align*}
then the equivalence classes must be close too:
\begin{equation}\label{eq::stabChol}
 \rho([X],[\tilde{X}]) \leq \frac{\sigma_{d}(X)}{\sqrt{2}}\left(1 - \sqrt{1 - \frac{2\norm{G - \tilde{G}}}{\sigma_{d}^{2}(X)}} \right),
\end{equation}
where $\rho$ is the distance between equivalence classes defined in~\eqref{eq::distance}.
\end{theorem}
Crucially, notice that for small $\norm{G - \tilde{G}}$, we have
$1 - \sqrt{1 - \frac{2\norm{G - \tilde{G}}}{\sigma_{d}^{2}(X)}} 
\approx \frac{1}{\sigma_{d}^{2}(X)}\norm{G - \tilde{G}}$.
Hence, the right-hand side of~\eqref{eq::stabChol} behaves like $\frac{1}{\sqrt{2}\sigma_{d}(X)}\norm{G - \tilde{G}}$, which by the discussion above cannot be improved.

\begin{proof}
Let $U\Sigma V^{T}$ be the SVD of $X\tilde{X}^{T}$. 
The orthogonal matrix $Q = UV^{T}$ optimally aligns $X$ and $\tilde X$, in the sense that $\rho([X], [\tilde X]) = \norm{X - Q\tilde X}$. Then,
\begin{align*}
	X\tilde{X}^{T}Q^{T} &= U\Sigma U^{T} = Q\tilde{X}X^{T}.
\end{align*}
Since the theorem statement depends on $X$ and $\tilde X$ only through $[X]$ and $[\tilde X]$, without loss of generality, suppose that $X$ and $\tilde{X}$ are already rotationally aligned, that is, $Q = I$. Then, $X\tilde{X}^T$ is symmetric, positive semidefinite. Define $\Delta X  = \tilde{X} - X$: notice that $\Delta X {X}^{T}$ is also symmetric, and
\begin{align*}
\tilde{X}^{T}\tilde{X} &= (X + \Delta X)^{T}(X + \Delta X) = X^{T}X + (\Delta X^{T}X + X^{T} \Delta X) + \Delta X^{T}\Delta X.
\end{align*}
Rearranging, we get:
\begin{align*}
\tilde{X}^{T}\tilde{X} - X^{T}X - \Delta X^{T}\Delta X &=  \mathcal{L}_{X}(\Delta X),
\end{align*}
where $\mathcal{L}_{X}$ is the operator defined in~\eqref{eq::operator}.
Since $\Delta X$ is in the subspace $\mathrm{H}_X$ defined in~\eqref{eq::horizontal}, it is also orthogonal to the null space $\mathrm{V}_X$ of $\mathcal{L}_{X}$.
As a result, we may write
 \begin{align*}
 \Delta X = \mathcal{L}_{X}^{\dagger}\left( \tilde X^{T} \tilde X - X^{T}X - \Delta X^{T}\Delta X \right),
 \end{align*}
where $\mathcal{L}_{X}^{\dagger}$ is the Moore--Penrose pseudo-inverse of $\mathcal{L}_{X}$.
Proposition~\ref{lem::operatorInverse} then implies:
 \begin{align*}
\norm{\Delta X} \leq \frac{1}{\sqrt{2}\cdot\sigma_{d}(X)}\left[ \norm{\tilde{X}^{T}\tilde{X}  - X^{T}X} + \norm{\Delta X}^{2}\right].
\end{align*}
Reorganizing, we get the following inequality:
 \begin{equation}\label{eq::feasible}
0 \leq  \norm{\Delta X}^{2} - \sqrt{2}\cdot\sigma_{d}(X)  \norm{\Delta X} + \norm{\tilde{X}^{T}\tilde{X} - X^{T}X}.
 \end{equation}
The right-hand side of this inequality is a 
quadratic in $\norm{\Delta X}$. 
Under our assumptions, the two roots of this quadratic, $\xi_{-}$ and $\xi_{+}$, are real and nonnegative: 
 \begin{align*} 
 \xi_{\pm} = \frac{\sqrt{2}\cdot\sigma_{d}(X)}{2} \pm \frac{\sqrt{ 2\cdot\sigma_{d}(X)^{2} - 4\cdot\norm{\tilde{X}^{T}\tilde{X} - X^{T}X}}}{2}.
 \end{align*}
 Since $\norm{\Delta X}$ satisfies \eqref{eq::feasible}, it lies outside the open interval defined by $(\xi_{-},\xi_{+})$. This means there are two possibilities: either $\norm{\Delta X} \in [0, \xi_-]$, or $\norm{\Delta X} \geq \xi_+$. We aim to exclude the latter. To do so, notice that Lemma~\ref{lem:TuLemma} together with our proximity assumption on the Gram matrices implies:
 \begin{align*}
  \norm{\Delta X} = \rho([X], [\tilde X]) & \leq \frac{1}{\sigma_{d}(X)}\frac{1}{\sqrt{ 	2(\sqrt{2} - 1)}} \frac{\sigma_{d}^{2}(X)}{2} < \frac{\sqrt{2}\cdot\sigma_{d}(X)}{2} \leq \xi_{+}.
 \end{align*}
 This allows to conclude that $\rho([X], [\tilde X]) = \norm{\Delta X} \leq \xi_-$. Upon factoring out $\frac{\sigma_{d}(X)}{\sqrt{2}}$ in the expression for $\xi_-$, this completes the proof.
\end{proof}

\subsection{Upper bounds on cloud estimation error}

Theorem~\ref{thm::cholesky} quantifies how a good estimator for the Gram matrix of $[X]$ can be turned into a good estimator for $[X]$ itself. In this part, we first show quantitatively that, with high probability, we can indeed have a good estimator for the Gram matrix. Afterwards, we connect this result with the above theorem to produce a bound on the estimation error of $[X]$.

The first result relies on standard concentration bounds for quadratic forms of Gaussian random variables.
(We remark that it is possible to relax the assumptions to require subgaussian noise rather than Gaussian noise.) We assume $\sigma$ is known, and we use the notation $\normop{X} = \sigma_1(X)$ for the operator norm.
The projection to the set of positive semidefinite matrices of rank at most $d$ is necessary to apply Theorem~\ref{thm::cholesky}, and causes no difficulties in practice.
\begin{theorem}\label{thm::stabilityEst}
Let $Y_{1}, \ldots , Y_{N}$ be i.i.d.\ observations drawn from model~\eqref{eq::description}
and let $\hat{G}_N$ be the Gram estimator as defined in~\eqref{eq:hatGN}.
Since the true Gram matrix $G = X^T X$ is positive semidefinite with rank at most $d$, project $\hat G_N$ to that set; in the notation of Algorithm~\ref{alg:estimation}:
\begin{align*}
	\tilde{G}_{N} = \sum_{i = 1}^{d} \alpha_i^2 v_i^{} v_i^T \in \underset{H\succeq 0 : \mathrm{rank}(H) \leq d}{\mathrm{argmin}} \normop{ \hat{G}_N - H}.
\end{align*}
Then, for any $\delta \in (0, 1)$, with probability at least $1 - \delta$:
\begin{align*}
\norm{\tilde{G}_{N} - G } 
\leq 8\sqrt{2d}\left[  \sqrt{\left( \frac{2\normop{X}^{2}\sigma^{2}+ d\sigma^{4}}{N}\right) k\log\left(\frac{10}{\delta}\right)   } 
+ \frac{\sigma^{2}}{N}k\log\left(\frac{10}{\delta}\right)  \right].
\end{align*}
\end{theorem}
\begin{proof}
For all $u \in \mathbb{S}^{k-1}$ (the unit sphere in $\mathbb{R}^k$), consider:
\begin{align}
u^{T} (\hat{G}_{N} - G)u
&= u^{T} \left(\frac{1}{N}\sum_{i=1}^{N} Y_{i}^{T}Y_{i}^{} - d \sigma^{2} I_{k} - X^{T}X \right)u\nonumber\\
&= \frac{1}{N} \sum_{i=1}^{N} \left[ \|(Q_iX + \sigma E_{i})u\|^{2} - (\|Xu\|^{2} + d\sigma^{2})\right]. \label{sumOperator}
\end{align}
Notice that $\|(Q_iX + \sigma E_{i})u\|^{2} = \|(X + \sigma Q_i^T E_{i})u\|^{2}$ is equal in distribution to $\|(X + \sigma E_{i})u\|^{2}$ since entries of $E_i$ are standard Gaussian. Thus, the first part of~\eqref{sumOperator} is distributed like a sum of squared norms of i.i.d.\ non-centered Gaussian vectors.
Reorganizing standard concentration bounds for noncentral $\chi^{2}$ (see for instance~\citep[ Lemma~8.1]{Birge01Lepski}) gives, for all $u$ in $\mathbb{S}^{k-1}$,
\begin{align}
\mathbb{P}\left[u^{T} (\hat{G}_{N} - G)u \geq 2\sqrt{\left(2\|Xu\|^{2} + d\sigma^{2}\right)\frac{t}{N}\sigma^{2}} + 2 \frac{t}{N}\sigma^{2}\right] & \leq e^{-t}, \textrm{ and }\\
\mathbb{P}\left[u^{T} (\hat{G}_{N} - G)u \leq -2\sqrt{\left(2\|Xu\|^{2} + d\sigma^{2}\right)\frac{t}{N}\sigma^{2}}\right] & \leq e^{-t}.
\end{align}
We then cover the sphere $\mathbb{S}^{k-1}$ with an $\varepsilon$-net. A union bound over all the elements of the net  
(see for instance~\citep[ Section 5.2.2]{Ver10}) yields
$$\mathbb{P}\left[\normop{\hat{G}_N - G } \leq 4\sqrt{(2\normop{X}^{2} + d\sigma^{2})\frac{t}{N}\sigma^{2}} + 4\frac{t}{N}\sigma^{2}\right] \geq 1 - 2 \cdot 5^{k} \cdot e^{-t}.$$
Since $G$ is positive semidefinite and has rank $d$, the projection $\tilde{G}_{N}$ satisfies:
\begin{align*}
\normop{\hat{G}_{N} - \tilde{G}_{N}} &\leq \normop{\hat{G}_{N} - G}.
\end{align*}
By the triangle inequality, we obtain:
\begin{align*}
\normop{G - \tilde{G}_{N}} &\leq \normop{G - \hat{G}_{N}} + \normop{\hat{G}_{N} - \tilde{G}_{N}} \leq 2 \normop{G - \hat{G}_{N}}.
\end{align*}
Since $\mathrm{rank}(G), \mathrm{rank}(\tilde{G}_{N}) \leq d$, it follows that $\mathrm{rank}(G - \tilde{G}_{N}) \leq 2d$ and we get:
\begin{align*}
\norm{G - \tilde{G}_{N}} \leq 2\sqrt{2d} \normop{G - \hat{G}_{N}}. 
\end{align*}
As a result,
$$\mathbb{P} \left[\norm{\tilde{G}_{N} - G } \leq 8\sqrt{2d(2\normop{X}^{2} + d\sigma^{2})\frac{t}{N}\sigma^{2}} + 8\sqrt{2d}\frac{t}{N}\sigma^{2} \right] \geq 1 - 2 \cdot 5^{k} \cdot e^{-t}.$$
Taking $t =  \log\left(\frac{2 \cdot 5^{k}}{\delta}\right)$  implies the final result:
$$
\norm{\tilde{G}_{N} - G } \leq 8\sqrt{2d}\left[  \sqrt{\left( 2\normop{X}^{2} + d\sigma^{2} \right) \frac{\sigma^{2}}{N}k\log\left(\frac{10}{\delta}\right)  } + \frac{\sigma^{2}}{N}k\log\left(\frac{10}{\delta}\right)  \right]
$$
with probability at least $1- \delta$.
\end{proof}

Combining Theorem~\ref{thm::stabilityEst} with a stability result on the thin Cholesky decomposition gives the main result on the stability of the proposed estimator, as measured with the distance $\rho$~\eqref{eq::distance}.
\begin{corollary}
Let $Y_{1}, \ldots , Y_{N}$ be $N$ i.i.d.\ samples drawn according to \eqref{eq::description}.
Let $[\tilde{X}]$ be the estimator  returned by Algorithm \ref{alg:estimation}.
Then, for large $N$, with probability at least $1 - \delta$,
\begin{align*}
\rho([X],[\tilde{X}]) \leq \frac{8L\sqrt{2d}}{\sigma_{d}(X)} \left[  \sqrt{\left( 2\normop{X}^{2} + d\sigma^{2} \right) \frac{\sigma^{2}}{N}k\log\left(\frac{10}{\delta}\right)  } + \frac{\sigma^{2}}{N}k\log\left(\frac{10}{\delta}\right)  \right],
\end{align*}
where $L$ can be taken as $1/\sqrt{2(\sqrt{2} - 1)}$ or, for large enough $N$, arbitrarily close to $1/\sqrt{2}$.
\end{corollary}
\begin{proof}
The first claim follows from Theorem~\ref{thm::stabilityEst} and Lemma~\ref{lem:TuLemma}.
For $N$ large enough,
Theorem~\ref{thm::stabilityEst} shows that the assumption of Theorem~\ref{thm::cholesky}, namely, $\norm{\tilde{G}_N - G} \leq \frac{\sigma_{d}^{2}(X)}{2}$, is satisfied with high probability. In that scenario, combining the two theorems
yields the second result.
\end{proof}

%% file: LowerBound.tex
\section{Statistical optimality of the estimator} \label{sec:statisticaloptimality}

Estimating the equivalence class of $X$ from samples of the form \eqref{eq::description} is a particular instance of an estimation problem under a group action.
Bandeira et al.~\citep{Bandeira2018EstimationGroupAction} showed that the statistical complexity
of estimation problems under a group action is connected to the  structure of the group acting on the parameter. 
In particular, it is shown that the minimum number of samples required to reliably estimate the parameter grows as $O(\sigma^{2p})$ where $p$ 
is the smallest degree of invariant polynomials required to fully characterize the equivalence classes.
We have shown in Section~\ref{invariant} that, in our case, $p = 2$.
In this section, we build on those results to show minimax lower bounds on the estimation of the equivalence class $[X]$. 
We also provide matching upper bounds, hence showing the statistical optimality of our estimator in the low SNR regimes.

To fix scale and to avoid pathological cases, throughout this section we assume that $X$ belongs to the space
\begin{equation}\label{def::searchspace}
\mathcal{X} = \{X \in \mathbb{R}^{d \times k}: \norm{X}^2 \leq d \textrm{ and } \sigma_{d}(X) \geq \eta \},
\end{equation}
where $\sigma_{d}(X)$ is the $d$th (that is, smallest) singular value of $X$, and $\eta > 0$ is fixed.

\subsection{Lower bound on the estimation error for high noise regimes}

In the presence of large noise (that is, for large $\sigma$), the MSE of \emph{any} estimator of the equivalence class of $X$ scales with $\sigma$ as $\sigma^4$: we make this precise in the following theorem.
\begin{theorem}\label{thm::minimax}
Suppose we observe $N$ samples $Y_{1},\ldots,Y_{N}$ drawn independently according to \eqref{eq::description}. Then the so-called minimax risk for estimating the orbit of $X$ satisfies
$$
\inf_{\hat{X}}\sup_{X \in \mathcal{X}}\mathbb{E}[\rho^{2}(X,\hat{X})] \asymp \dfrac{\sigma^{4}}{N},
$$
for sufficiently large $\sigma$, where $\rho$ is as in~\eqref{eq::distance} and $\mathcal{X}$ is defined by~\eqref{def::searchspace}, and the infimum is taken over all possible estimators, random or deterministic.
\end{theorem}
The general study of lower bounds for estimation under a group action has been addressed by Abbe et al~\citep{AbbePereiraSinger18EstimationGpActionLB}, using Chapman--Robbins bounds.
Here, we take a different approach, 
 proving minimax rates using two ingredients: a \textit{tight} bound on the Kullback--Leibler (KL) divergence
and a bound on the packing number of a particular metric space.
However, since we are not aware of any result on the packing number of our parameter space for the metric $\rho$, we consider a strict subset of this parameter space for which tight bounds
 on the packing number are known.
Specifically, we start by noticing that the Grassmannian $\Gkd$---the set of $d$ dimensional subspaces of $\R^k$---is in correspondence with a subset of the parameter space.
Building on this observation, we restrict ourselves to the strictly simpler problem where the matrix $X$ satisfies the condition $XX^{T} = I_{d}$. 
In this restricted setting, there is a one-to-one correspondence between an equivalence class and an element of the Grassmannian.
We then use a result on the covering number of the Grassmannian to control its \textit{local} packing number. This result was used by Cai et al.~\citep{CaiMaWu13RatesSparsePCA} in the context of 
optimal rates of estimation  for the principal subspace of a covariance matrix under a sparsity assumption.
Then, we show a tight bound of the KL divergence.
This bound is a particular case of a more general result on the KL divergence of samples observed under the action of a group.
\begin{enumerate}
\item Packing number of the Grassmannian:
we start by defining a metric on $\Gkd$. This metric on the Grassmanian is shown to be equivalent to the distance between equivalence
classes $\rho$.
We then show tight lower and upper bounds on the covering number of $\Gkd$ for the aforementioned metric. This result is due to~\TODO{Szarek} \citep{Szarek98NetsGrassman}. 
We then use a result on the local packing number of the Grassmannian: leveraging the previous result, for any $\alpha \in (0,1)$ and for any $\varepsilon$ small enough, we give a lower bound on the $\alpha\varepsilon$-packing number of a ball of radius $\varepsilon$. This follows the technique proposed by Yang and Barron~\citep{YangBarron99Minimax}.
\item Tight control of the KL divergence: 
we state a lemma giving a bound on the $\mathrm{KL}$ divergence of the distribution  
$\mathbb{P}_{X}$ of samples drawn according to \eqref{eq::description}.
This lemma follows from a result by Bandeira et al.~\citep{Bandeira2018EstimationGroupAction}.
\end{enumerate}
%
We start by giving the result on the covering number of the Grassmannian.
\begin{lemma}\label{MetricEntropy}[Cai et al.~\citep{CaiMaWu13RatesSparsePCA}, Lemma 1]
Define the metric on $\Gkd$ by $\tilde{\rho}(V,U) = \norm{V^{T}V - U^{T}U}$.
Then, for any $\varepsilon \in \left(0,\sqrt{2\min(d,k - d)}\right]$, we have
$$
\left(\dfrac{c_{0}}{\varepsilon}\right)^{d(k - d)} \leq \mathfrak{N}(\Gkd,\varepsilon) \leq \left(\dfrac{c_{1}}{\varepsilon}\right)^{d(k - d)},
$$
where $\mathfrak{N}(\Gkd,\varepsilon)$ is the $\varepsilon$-covering number of $\Gkd$ with respect to the metric $\tilde{\rho}$ and $c_{0}$, $c_{1}$ are absolute constants.
\end{lemma}
%
Building on the previous result, we can state a result on the \textit{local} packing of $G(k,d)$.
\begin{lemma}\label{lem::localPacking}
Let $B(V,\varepsilon) = \{U \in \Gkd: \tilde{\rho}(U,V) \leq \varepsilon \}$,
 $\alpha \in (0,1)$ and $\varepsilon \in (0,\varepsilon_{0}]$.
Then, there exists $V^{*} \in \Gkd$ such that: 
$$
\mathfrak{M}(B(V^{*},\varepsilon),\alpha\varepsilon) \geq \left(\dfrac{c_{0}}{\alpha c_{1}}\right)^{d(k - d)},
$$
where $\mathfrak{M}(\Gkd,\varepsilon)$ is the $\varepsilon$-packing number of $\Gkd$ with respect to the metric $\tilde{\rho}$
and $c_{0}$, $c_{1}$ are the absolute constants in Lemma~\ref{MetricEntropy}.
\end{lemma}
\begin{proof}
The original proof can be found in~\citep{YangBarron99Minimax}. For convenience, we provide the proof in Appendix~\ref{appendix::minimax}.
\end{proof}

Finally, we state a bound on the KL divergence of the distribution 
$\mathbb{P}_{X}$ of samples drawn according to \eqref{eq::description}.
\begin{lemma}\label{LemmaKLBound}
For given dimensions $d$ and $k$,  there exists a universal constant $C$ such that, for any $X_1$, $X_2 \in \R^{d \times k}$  
with $\rho(X_{1},X_{2}) \leq \dfrac{\norm{X_{1}}}{3}$
and for any  $\sigma > 1$, we have:
$$
\KL{\mathbb{P}_{X_{1}} || \mathbb{P}_{X_{2}}} \leq C\sigma^{-4}\rho^{2}(X_{1},X_{2}).
$$
\end{lemma}
\begin{proof}
See Appendix \ref{KLBound}.
\end{proof}

Combining these three results, 
we obtain Theorem \ref{thm::minimax}.
The proof technique is inspired from~\citep{YangBarron99Minimax} and~\citep{CaiMaWu13RatesSparsePCA}. The full proof can be found in Appendix~\ref{appendix::minimax}.
The constant in the lower bound exhibits a dependence in the dimension, through the metric entropy of the Grassman manifold.
As we focus on the dependence in $\sigma$ and $N$ (noise level and number of samples), we leave the question of whether this dependence in the dimension is tight for future research.


\subsection{Lower bound on the sample complexity for high noise regimes}

We now show that, still when $\sigma$ is large, the number of samples drawn according to~\eqref{eq::description} necessary to reliably estimate the equivalence class of 
$[X]$ grows as $O(\sigma^{4})$.
For that matter, we provide matching upper and lower bounds:
the upper bound on the KL divergence obtained in Lemma~\ref{LemmaKLBound} in combination with Neyman--Pearson's lemma (see [~\citep{RigolletNotes}, Lemma 4.3] for instance) yields the lower bound, while properties of the estimator obtained with Algorithm~\ref{alg:estimation} give the upper bound. This result is similar in its message to, but technically different from, Theorem~\ref{thm::minimax}.


\begin{theorem}
	Let $\tau \in (0, 1/2)$ be arbitrary. There exists a constant $\tilde c$ (possibly function of $\tau$) such that the following holds: let $X_1, X_2 \in \mathcal{X}$ belong to two distinct equivalence classes, and let $Y_1, \ldots, Y_N$ be drawn according to~\eqref{eq::description}, where $X$ is either $X_1$ or $X_2$, with probability 1/2 each. Any test $\psi$ whose task it is to decide whether $X$ is $X_1$ or $X_2$ has probability of error at least $1/2-\tau$ whenever $N < \tilde{c}\sigma^{4}$.
	
Moreover, there exists a constant $\tilde{C}$ such that Algorithm~\ref{alg:estimation} outputs $\hat{X}_{N}$ satisfying 
$$\mathbb{P}\left(\rho(\hat{X}_{N},X) \leq \varepsilon\right) \geq 1 - \delta$$
whenever $N \geq \tilde{C}\dfrac{\sigma^{4}}{\delta\cdot\varepsilon^{2}}.$
\end{theorem}

\begin{proof}
We start by proving that no statistical procedure can reliably distinguish between two equivalence classes whenever the
number of samples is less than $O(\sigma^{4})$. 
In particular, for any test $\psi$ using $N$ samples we have:
\begin{align*}
\mathbb{P}_{X_{1}}(\psi = 2) + \mathbb{P}_{X_{2}}(\psi = 1) &\geq 1 - \TV{\mathbb{P}_{X_{1}}^{N} || \mathbb{P}_{X_{2}}^{N}}\\
&\geq 1 - \sqrt{\dfrac{1}{2}\KL{\mathbb{P}_{X_{1}}^{N} || \mathbb{P}_{X_{2}}^{N}}}\\
&\geq 1 - \sqrt{\dfrac{1}{2}N\cdot\KL{\mathbb{P}_{X_{1}} || \mathbb{P}_{X_{2}}}}\\
&\geq 1 - \sqrt{\dfrac{CN\rho^{2}(X_{1},X_{2})}{2\sigma^{4}}},
\end{align*}
where the first inequality follows from Neyman--Pearson's Lemma while the second follows from Pinsker's inequality.
Therefore, if $N \leq \dfrac{8}{C\rho^{2}(X_{1},X_{2})}\tau^{2}\sigma^{4}$, we get that: 
\begin{align*}
\frac{1}{2}\mathbb{P}_{X_{1}}(\psi = 2) + \frac{1}{2}\mathbb{P}_{X_{2}}(\psi = 1) &\geq \frac{1}{2} - \tau,
\end{align*}
yielding the sought lower bound on the probability of error. The constant $\tilde c$ can be set uniformly against the choice of $X_1, X_2$ since $\mathcal{X}$ is bounded.

For the upper bound, we simply notice that, for large $\sigma$, by Markov's inequality,
\begin{align*}
\mathbb{P}\left(\rho(X,\hat{X}) \geq \varepsilon\right) &\leq \dfrac{1}{\varepsilon^{2}}\mathbb{E}\left[\rho^{2}(X,\hat{X})\right] 
\leq \dfrac{\tilde{C}}{\varepsilon^{2}} \dfrac{\sigma^{4}}{N},
\end{align*}
where the first inequality follows from Markov's inequality while the second inequality follows from the computations in Appendix~\ref{EstimatorMSE}. 
\end{proof}

\subsection{A comment regarding the special orthogonal group}

In our model, we consider the case where the orthogonal transformations acting on the cloud are in $\Od$. In some applications, it may be more appropriate to assume rotations $Q_1,\ldots, Q_N$ in the special orthogonal group 
$\mathcal{SO}(d)$, that is, with determinant $+1$ (no reflections).
However, the problem in $\mathcal{SO}(d)$ is harder. Indeed, by the  
first fundamental theorem of invariant theory  for the special orthogonal group,
the entries of the Gram matrix $X^{T}X$  and the $d \times d$ minors of $X$ generate the ring of invariant polynomials (see~\citep[Prop.~10.2]{kraft2000classical} for instance). 
Since our algorithm only recovers the Gram matrix, that is, the orbits of matrix $X$ up to a reflection, it falls short of estimating such reflection.
Estimating the reflection can be done by estimating the determinants of the submatrices of $X$ of size $d \times d$, which requires $O(\sigma^{2d})$ samples at high noise level, as further discussed in~\citep{Bandeira2018EstimationGroupAction}.



%% file: NoiseRegimes.tex
\subsection{Noise regimes}\label{sec::noiseRegimes}



We here summarize the adaptive optimality property of our estimator, characterized by a phase transition around a critical noise level.


\begin{proposition}\label{prop::regimes}
The estimator in Algorithm~\ref{alg:estimation} exhibits a phase transition around a critical noise level, namely:
\begin{itemize}
\item[--] If $\sigma^{2} \gg 1$, the MSE of the estimator behaves like $O\left(\sigma^{4}/N\right)$, which is optimal by Theorem~\ref{thm::minimax}.
\item[--] If $\sigma^{2} \ll 1$, the MSE of the estimator behaves like $O\left(\sigma^{2}/N\right)$, which is optimal
given the fact that the problem is strictly harder than estimating the matrix $X$ provided the rotations are available, 
in which case the lower bound on the estimation is $O(\sigma^{2}/N)$.
\end{itemize}
\end{proposition}
\begin{proof}
The proof of the upper bound relies on a computation of bounds on the MSE of the estimator proposed in
Algorithm~\ref{alg:estimation}: see Appendix~\ref{appendix::MSE}. 
\end{proof}

%% file: Numerics.tex
\section{Numerical experiments} \label{sec:XP}

In this section, we provide numerical support for our theoretical predictions.\footnote{Code to generate the figures: \url{https://github.com/thomaspdl/GeneralizedProcrustes}.}
In particular, we highlight the adaptative optimality of Algorithm~\ref{alg:estimation} 
and validate the heuristic proposed for the estimation of $\sigma^2$. 

\subsection{Noise regimes}\label{NumericsNoiseRegimes}

Figure~\ref{fig::DonohoPlots} illustrates the accuracy of Algorithm~\ref{alg:estimation} over a large range of values of $N$ (number of observations) and $\sigma$ (noise level), with $k = 100$ and $d = 3$ fixed. A clear phase transition is visible, delineating a regime where our estimator is accurate, and one where it fails. 
Crucially, the phase transition illustrates the adaptiveness of the estimator. 
Specifically, as predicted, for small $\sigma$, the number of observations needs to grow as $N \sim \sigma^2$ in order to preserve a constant MSE while compensating for noise. For large $\sigma$, this relationship deteriorates and we require $N \sim \sigma^4$. As explained in Section~\ref{sec:statisticaloptimality}, this is not merely a requirement of our estimator: any estimator is subject to the same needs, and it is a positive feature of Algorithm~\ref{alg:estimation} that it adapts automatically to both regimes.

To illustrate these two noise regimes, the graph is overlaid with two lines.
The red line of slope $1/2$ represents a noise level varying as $O(\sqrt{N})$.
Its intercept is chosen as follows:
consider an oracle which knows the rotations $Q_{i}$ affecting the measurements $Y_{i}$.
This oracle can compute the MLE simply by undoing the rotations,
then averaging.
It is easy to see that the MSE of that oracle is $\frac{\sigma^{2}dk}{N}$ in Frobenius distance (it may be slightly less in $\rho$ distance).
The red line shows the relationship between $N$ and $\sigma$ when that oracle has a relative MSE of $0.95$. It is interesting to see that, for low noise levels, the invariants-based estimator has performance similar to that oracle. The blue line of slope $1/4$ represents a noise level varying as $O(N^{\frac{1}{4}})$: its intercept is chosen manually, for illustration.


\begin{figure}[p]
 \begin{center}
 \includegraphics[width=.8\linewidth]{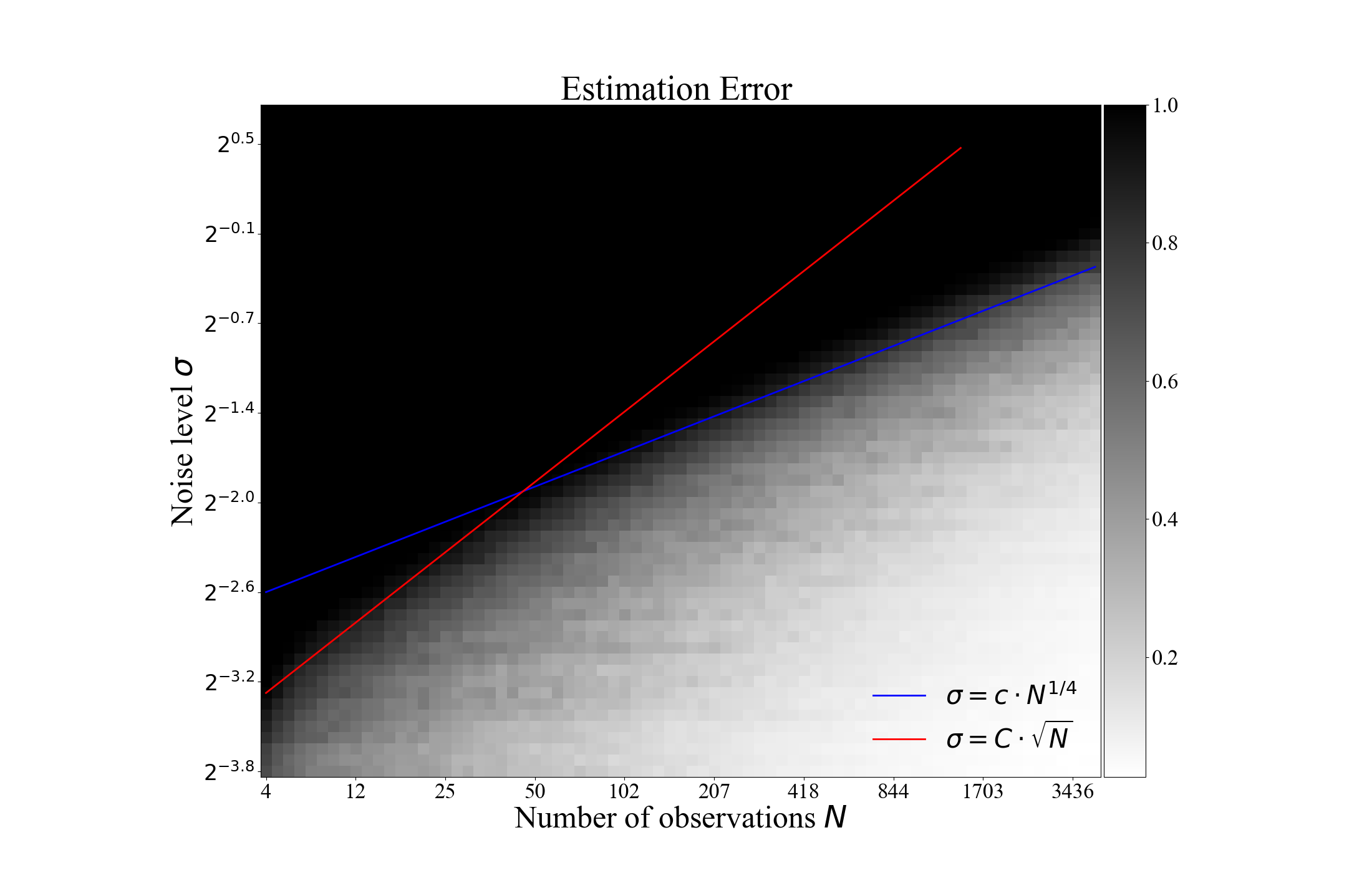}
 \end{center}
 \caption{We generate a random cloud $X \in \mathbb{R}^{d \times k}$ with $d = 3$ and $k = 100$ and i.i.d.\ Gaussian entries of variance 1. The cloud is then renormalized to have unit Frobenius norm. For each pair $(N, \sigma)$ on a log-log grid, $N$ observations of $X$ are produced with noise level $\sigma$ (known), and the equivalence class $[X]$ is estimated using Algorithm~\ref{alg:estimation}. Each pixel's brightness indicates the relative estimation error $\rho([X], [\hat X])/\norm{X}$~\eqref{eq::distance}, capped at one and then averaged over 50 independent repetitions. The blue and red lines illustrate how Algorithm~\ref{alg:estimation} is adaptive to noise levels: see Section~\ref{NumericsNoiseRegimes} for details.}
 \label{fig::DonohoPlots}
 \end{figure}

 
 \subsection{Estimation of $\sigma^2$}\label{NumericsSigma}
 
 To illustrate the accuracy of the approximation made in~\eqref{eq::SigmaEstimator},
we compute the mean empirical error of our estimator. 
 In particular, for different numbers of observations $N$, we estimate the average relative error of the estimator reported in Fig.~\ref{fig::BiasSigma}.
 We observe a small empirical relative error decreasing with the number of samples. 
\begin{figure}[h]
 \begin{center}
 \includegraphics[width=.6\linewidth]{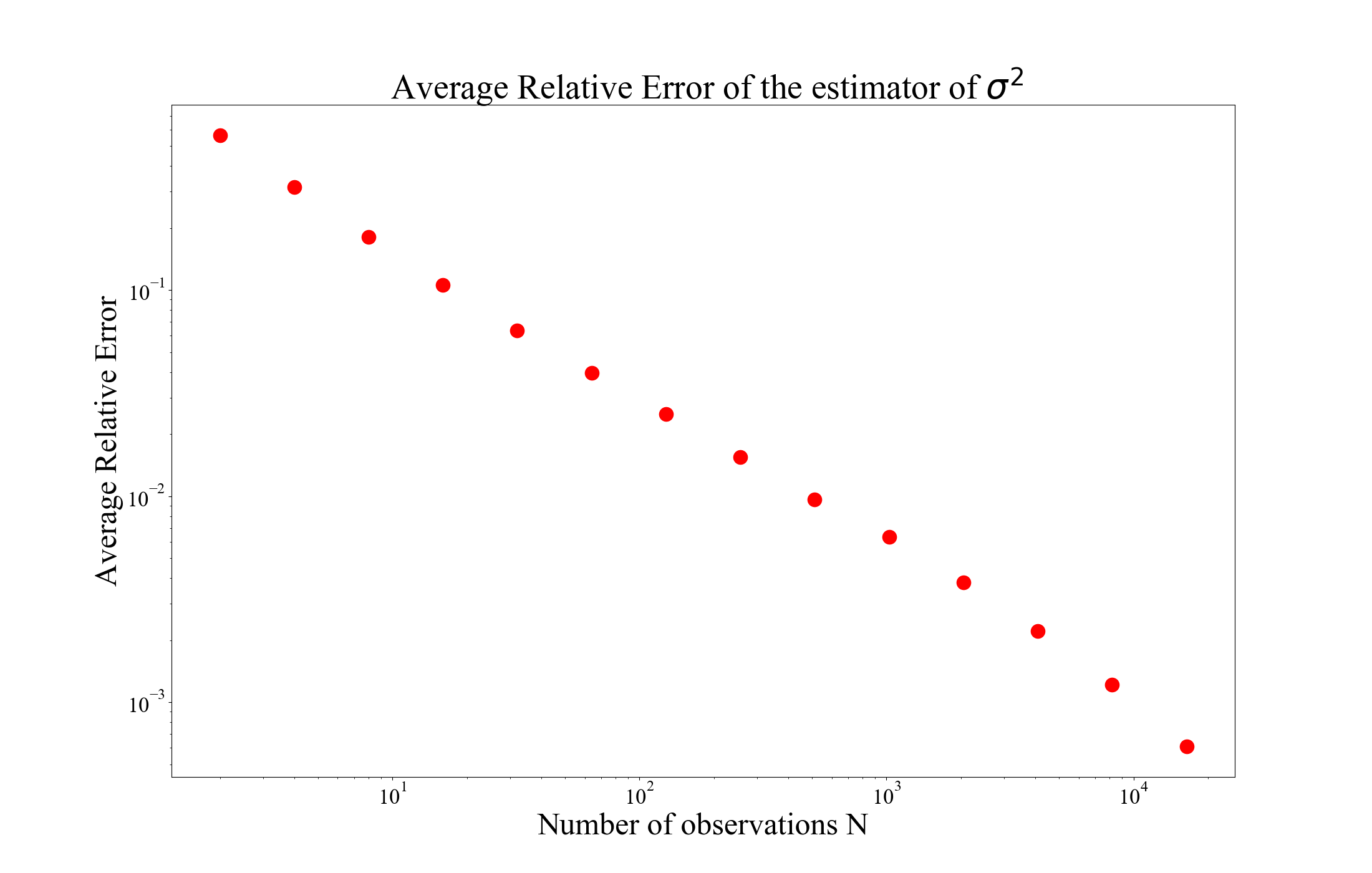}
 \end{center}
 \caption{
We generate a random cloud $X \in \mathbb{R}^{d \times k}$ with $d = 3$ and $k = 100$ and i.i.d.\ Gaussian columns.
For different numbers of samples, we compute the relative error $\frac{|\sigma^2 - \hat{\sigma}^2|}{\sigma^2}$ on the estimation of $\sigma^2$. This relative error is then averaged over 250 i.i.d.\ realizations. The results are displayed in log-log scale.
In practice, we notice that the empirical variance of the estimator is dominated by the relative error of estimation.
  }
 \label{fig::BiasSigma}
 \end{figure}

%% file: Conclusion.tex
\section{Perspectives}


Throughout this paper, we only consider the case where one cloud of points must be estimated.
An interesting variant is \textit{heterogeneous} Procrustes, where each observation consists of a noisy, rotated version
of a cloud of points picked at random among $K$ possibilities, $X_1, \ldots, X_K$; for example:
\begin{align*}
	Y & = QX_s + \sigma E, \textrm{ with }  \mathbb{P}(s = j) = w_j,
\end{align*}
where $w_1, \ldots, w_K \geq 0$ sum to 1.
Several approaches to solve this problem are possible: taking a moments-based method similar to the one proposed by Hsu and Kakade~\citep{hsu2013learning} or inverting invariant features using a nonconvex optimization approach as in~\citep{boumal2017heterogeneous,ma2018mra2D} could be two interesting directions. Regarding the latter, one idea is to consider these fourth-order $\mathcal{O}(d)$-invariants: $Y^T Y \otimes Y^T Y$, where $\otimes$ is the Kronecker product.

Another possible extension is to consider observations with \emph{unlabeled} clouds of points, that is: we observe several copies of $X$ after an unknown rotation, and also after an unknown permutation of the points to model the fact that we do not know which point is which across various observations. This can be modeled as
\begin{align*}
	Y & = QXP + \sigma E,
\end{align*}
where $Q$ is an unknown orthogonal matrix of size $d \times d$ as usual, and $P$ is an unknown permutation matrix of size $k \times k$. Assuming i.i.d.\ Gaussian noise, in distribution, $Y = Q(X + \sigma E)P$, so that the singular values of $Y$ are equal to those of $X + \sigma E$: up to the noise, the singular values are (non-polynomial) invariants. Studying the distribution of the singular values of $Y$ as a function of those of $X$ may allow one to estimate the general shape of the cloud $X$ (specifically, its singular values) without the need to register clouds (no rotation estimation, and no point correspondence estimation). This connects to principal component analysis.

Yet another extension is to consider \emph{projected} observations. For example, $X$ is a cloud of points in three dimensions ($d = 3$), but observations are of the form $Y = P(QX + \sigma E)$, where $P$ is a projector to a two-dimensional plane (for example, the camera plane). Equivalently,
\begin{align*}
	Y & = U^T X + \sigma E,
\end{align*}
where $U$ is an unknown matrix of size $3 \times 2$ with orthonormal columns, and $E$ has size $2 \times k$. Assuming $U$ is uniformly distributed, $\mathbb{E}[UU^T] = \frac{2}{3}I_3$,
so that $Y^TY$, while not an invariant, does reveal information about $X$ in expectation since $\mathbb{E}[Y^T Y] = \frac{2}{3}X^T X + d\sigma^2 I_k$. Much of the machinery developed in the present paper applies directly to this extended setting: it would be interesting to study it further.

In all of these, one could also include the possibility that clouds are not centered, that is: they are observed only after an unknown translation, on top of other transformations.


%
%

\section*{Acknowledgments}

We thank Afonso Bandeira, Adrien Bilal, Jianqing Fan, Joe Kileel and Joao Pereira for insightful discussions.

\section*{Funding}

NB is partially supported by NSF award DMS-1719558.

\section*{Data Availability Statement} 

The data underlying this article are available on GitHub, at https://github.com/thomaspdl/GeneralizedProcrustes.


%% file: Distance.tex
\section{Procrustes distance}\label{distance}


Here, we show how to explicitly compute the distance $\rho$~\eqref{eq::distance}.
Since:
$$
\norm{X_{1} - QX_{2}}^{2} = \norm{X_{1}}^{2}  + \norm{X_{2}}^{2}  - 2\dotprod{X_{1},QX_{2}},
$$
finding the optimal $Q$ is equivalent to solving:
$$
\max_{Q \in \mathcal{O}(d)}\langle X_{1},QX_{2} \rangle = \max_{Q \in \mathcal{O}(d)}\Tr{QX_{2}X_{1}^{T}}.
$$
By SVD we can write $X_{1}^{}X_{2}^{T} = U\Sigma V^{T}$ and the solution of the problem is given by the polar factor of $X_{1}^{}X_{2}^{T}$:
\begin{equation}\label{eq::optrotation}
Q = UV^{T}.
\end{equation}

%% file: ImpossibilityEstimation.tex
\section{Impossibility of estimation for the rotations}\label{ImpossibilityEstimation}


Here, we give a proof of Proposition~\eqref{prop::ImpossibilityRotations}.
We start by reminding the statement of the proposition below.
\begin{proposition}
Let us consider the following hypothesis testing problem on the distribution of the unknown distribution $\mathbb{P}$ of the sample: 
\begin{align*}
& H_{1}: \mathbb{P}_{1} \sim \mathcal{N}(QX,\sigma^{2}I_{dk}),\\
& H_{2}: \mathbb{P}_{2} \sim \mathcal{N}(-QX,\sigma^{2}I_{dk}).
\end{align*}
Then for any test $\psi$ and for any precision $\delta > 0$, it is possible to find a sufficiently large noise level $\sigma_{0}$ such that, 
for any noise level $\sigma$ larger than $\sigma_{0}$, the sum of type I and type II errors is large:
\begin{align*}
\mathbb{P}_{1}(\psi(Y) = 2) + \mathbb{P}_{2}(\psi(Y) = 1)   &\geq 1 - 2\delta.
\end{align*}
Hence, if the two hypotheses are equally likely; the probability of error is at least $1/2 - \delta$.
\end{proposition}
\begin{proof}
The likelihood ratio test consists of studying the ratio of the densities $L(Y) = \dfrac{f_{1}(Y)}{f_{2}(Y)}$ 
given observation $Y$.
If $L(Y) > 1$, then $H_{1}$ is kept; otherwise, $H_{2}$ is kept.
Hence, $\mathbb{P}_{1}(\psi^{*} = 2) = \mathbb{P}_{1}(L(Y) < 1)$.
\begin{align*}
L(Y) &= \dfrac{f_{1}(Y)}{f_{2}(Y)} = \dfrac{e^{-\frac{1}{2\sigma^{2}} \norm{Y - QX} ^2}}{e^{-\frac{1}{2\sigma^{2}} \norm{Y + QX}^2 }} \\
&=  \dfrac{e^{-\frac{1}{2\sigma^{2}} (\norm{Y}^2 + \norm{X}^2 - 2 \dotprod{Y, QX}) }}
{e^{-\frac{1}{2\sigma^{2}} (\norm{Y}^2  + \norm{X}^2 + 2 \dotprod{Y, QX}) }} \\
&= e^{\frac{2}{\sigma^{2}}\dotprod{Y, QX}} = e^{\frac{2}{\sigma^{2}}\dotprod{QX + \sigma E , QX}} \\
&= e^{\frac{2}{\sigma^{2}}( \norm{X}^{2} + \sigma \dotprod{E , QX})}. 
\end{align*}
This yields:
$$L(Y) < 1 \iff \log(L(Y)) < 0 \iff \norm{X}^{2} + \sigma \dotprod{E, QX} < 0.$$
Yet,
\begin{align*}
\mathbb{P}(\norm{X}^{2} + \sigma \dotprod{E, QX} < 0) &= \mathbb{P}(\mu + \tilde{\sigma}\xi < 0), 
\end{align*}
with $\mu = \norm{X}^{2}$, $\tilde{\sigma} = \sigma \norm{X}$ and $\xi \sim \mathcal{N}(0,1)$.
Hence,
\begin{align*}
\mathbb{P}_{1}(\psi^{*} = 2)  &= \mathbb{P}(\mu + \tilde{\sigma}\xi < 0) \\
&= \mathbb{P}(\xi < -\dfrac{\mu}{\tilde{\sigma}}) \\
&= \mathbb{P}\left(\xi < -\dfrac{\norm{X}}{\sigma}\right) \\
&= F\left(-\dfrac{\norm{X}}{\sigma}\right),
\end{align*}
where $F$ is the cumulative density function of a standard Gaussian random variable.
Similarly, we get:
$$
\mathbb{P}_{2}\left(\psi^{*} = 1\right) = F\left(-\dfrac{\norm{X}}{\sigma}\right).
$$
Hence, by Neyman--Pearson's lemma~\citep{NeymanPearson33Test}, 
it is sufficient to take $\sigma$ such that $F\left(-\dfrac{\norm{X}}{\sigma}\right) = \dfrac{1 - 2\delta}{2}$
to get that for any test $\psi$:
$$
\mathbb{P}_{2}(\psi = 1) + \mathbb{P}_{1}(\psi = 2) \geq 1 - 2\delta,
$$
where $F$ is the cumulative density function of a standard normal distribution. 
\end{proof}

%% file: Cholesky.tex
\section{Stability of the Cholesky decomposition}\label{CholAppendix}

We give a proof of Proposition~\ref{lem::operatorInverse} which is used in the proof of Theorem~\eqref{thm::cholesky} about the stability of the estimator given by Algorithm~\ref{alg:estimation}.
\begin{proof}
Let $X = U\Sigma V^T$ be the thin SVD of $X$. Here, $V \in \R^{k \times d}$ has orthonormal columns, $U$ is an orthogonal matrix of size $d \times d$ and $\Sigma$ is a diagonal matrix of size $d$ with entries $\sigma_1 \geq \cdots \geq \sigma_d > 0$. It is always possible to pick $V_\perp$ (a complement of the orthonormal basis $V$) such that $\begin{bmatrix} V & V_\perp \end{bmatrix}$ is an orthogonal matrix of size $k \times k$. Hence, for any $\dot X \in \R^{d \times k}$, there exist matrices $A, B$ of appropriate size such that
\begin{align*}
	\dot X & = UAV^T + UBV_\perp^T.
\end{align*}
Using this parameterization,
\begin{align*}
	\mathcal{L}_X(\dot X) & = V\Sigma (AV^T + BV_\perp^T) + (VA^T + V_\perp B^T)\Sigma V^T \\
					& = \begin{bmatrix}
					V & V_\perp
					\end{bmatrix} \begin{bmatrix}
					\Sigma A + A^T \Sigma & \Sigma B \\ B^T \Sigma & 0
					\end{bmatrix} \begin{bmatrix}
					V & V_\perp
					\end{bmatrix}^T.
\end{align*}
We use this to derive an SVD of $\mathcal{L}_X$ restricted to $\mathrm{H}_X$, that is, for $\dot X$ such that $A\Sigma = \Sigma A^T$. To this end, consider the orthonormal basis of $\mathrm{H}_X$ composed of the following elements, and its image through $\mathcal{L}_X$ (below, $e_i$ denotes the $i$th column of the identity matrix of appropriate dimension as indicated by context):
\begin{enumerate}
	\item With $A = e_ie_i^T, B = 0$ for $1 \leq i \leq d$, \\ $\mathcal{L}_X(\dot X) = 2 \sigma_i \begin{bmatrix}
	V & V_\perp
	\end{bmatrix} \begin{bmatrix}
	e_ie_i^T & 0 \\ 0 & 0
	\end{bmatrix} \begin{bmatrix}
	V & V_\perp
	\end{bmatrix}^T$;
	\item With $A = \frac{\sigma_i e_i e_j^T + \sigma_j e_j e_i^T}{\sqrt{\sigma_i^2 + \sigma_j^2}}, B = 0$ for $1 \leq i < j \leq d$, \\ $\mathcal{L}_X(\dot X) = \sqrt{2(\sigma_i^2 + \sigma_j^2)} \begin{bmatrix}
	V & V_\perp
	\end{bmatrix} \begin{bmatrix}
	\frac{e_ie_j^T + e_je_i^T}{\sqrt{2}} & 0 \\ 0 & 0
	\end{bmatrix} \begin{bmatrix}
	V & V_\perp
	\end{bmatrix}^T$;
	\item With $A = 0, B = e_i e_j^T$ for $1 \leq i \leq d$ and $1 \leq j \leq k-d$, \\ $\mathcal{L}_X(\dot X) = \sqrt{2}\sigma_i \begin{bmatrix}
	V & V_\perp
	\end{bmatrix} \begin{bmatrix}
	0 & \frac{e_i e_j^T}{\sqrt{2}} \\ \frac{e_j e_i^T}{\sqrt{2}} & 0
	\end{bmatrix} \begin{bmatrix}
	V & V_\perp
	\end{bmatrix}^T$.
\end{enumerate}
To verify that the inputs yield an orthonormal basis of $\mathrm{H}_X$ as announced, check that each of these choices yields a matrix $\dot X$ in $\mathrm{H}_X$; they are indeed orthonormal; and they are in sufficient number to cover $\dim \mathrm{H}_X = dk - \dim \mathrm{V}_X = dk - \frac{d(d-1)}{2}$. As the outputs of $\mathcal{L}_X$ applied to the basis elements are also orthogonal, it is clear that the singular values of $\mathcal{L}_X$ restricted to $\mathrm{H}_X$ are:
\begin{enumerate}
	\item $2\sigma_1, \ldots, 2\sigma_d$;
	\item $\sqrt{2(\sigma_i^2 + \sigma_d^2)}$ for $1 \leq i < j \leq d$; and
	\item $\sqrt{2}\sigma_1, \ldots, \sqrt{2}\sigma_d$, each repeated $k-d$ times.
\end{enumerate}
(And of course, the singular values of $\mathcal{L}_X$ on $\mathrm{V}_X$ are zero, $d(d-1)/2$ times.)
\end{proof}

%% file: MinimaxProofs.tex
\section{Minimax lower bound}\label{appendix::minimax}

\subsection{Proof of Lemma \ref{lem::localPacking}}

\begin{proof}
Let us consider $G_{\varepsilon}$, a minimal $\varepsilon$-cover of $\Gkd$ for the semi metric $\tilde{\rho}$, i.e. such that:
\begin{align*}
\Gkd &= \underset{U \in \mathrm{G}_{\varepsilon}}{\bigcup} B(U,\varepsilon).
\end{align*}
Then:
\begin{equation}\label{contrad}
\mathfrak{N}(\Gkd,\alpha\varepsilon) = \mathfrak{N}\left(\underset{U \in G_{\varepsilon}}{\bigcup} B(U,\varepsilon),\alpha\varepsilon \right) 
\leq \sum_{U \in \mathrm{G}_{\varepsilon}}\mathfrak{N}(B(U,\varepsilon),\alpha\varepsilon).
\end{equation}
For contradiction, assume that for all $U$ in $G_{\varepsilon}$ we have 
\begin{align*}
\mathfrak{N}(B(U,\varepsilon),\alpha\varepsilon) 
&< \dfrac{\mathfrak{N}(\Gkd,\alpha\varepsilon)}{\mathfrak{N}(\Gkd,\varepsilon)} 
 = \dfrac{\mathfrak{N}(\Gkd,\alpha\varepsilon)}{|\mathrm{G}_{\varepsilon}|}.
 \end{align*}
This implies:
\begin{align*}
\mathfrak{N}(\Gkd,\alpha\varepsilon) = \mathfrak{N}\left(\underset{U \in G_{\varepsilon}}{\bigcup} B(U,\varepsilon),\alpha\varepsilon \right) 
> \sum_{U \in \mathrm{G}_{\varepsilon}}\mathfrak{N}(B(U,\varepsilon),\alpha\varepsilon),
\end{align*}
which contradicts~\eqref{contrad}.
Hence, there exists $U^{*} \in \mathrm{G}_{\varepsilon}$ such that:
$$
\mathfrak{N}(B(U^{*},\varepsilon),\alpha\varepsilon) \geq \dfrac{\mathfrak{N}(\Gkd,\alpha\varepsilon)}{\mathfrak{N}(\Gkd,\varepsilon)} 
\geq \left(\dfrac{c_{0}}{\alpha c_{1}}\right)^{d(k - d)},
$$
where the second inequality follows from Lemma \ref{MetricEntropy}.
The fact that
$\mathfrak{M}(E,\varepsilon) \geq  \mathfrak{N}(E,\varepsilon)$  for any $E$
allows to conclude the proof.
\end{proof}

\subsection{Proof of Theorem \ref{thm::minimax}}

We will use the following lemma, originally given in~\citep{Tuetal2016LowRankProcrustesFlow}.
\begin{lemma}\label{lem::tech1}
	For any $X_{2} \in \mathbb{R}^{d \times k}$ obeying $\ro{X_{2},X_{1}} \leq \dfrac{1}{4}\normop{X_{1}}$, we have
	$$
	\norm{X_{2}^{T}X_{2}^{} - X_{1}^{T}X_{1}^{}} \leq \dfrac{9}{4}\normop{X_{1}}\ro{X_{2},X_{1}}.
	$$
\end{lemma}

\begin{proof}
	For all $Q$ in $\mathcal{O}(d)$ we have
	\begin{align*}
	\norm{X_{2}^{T}X_{2}^{} - X_{1}^{T}X_{1}^{}}  &= \norm{X_{2}^{T}X_{2} - X_{2}^{T}QX_{1} + X_{2}^{T}QX_{1} - (QX_{1})^{T}QX_{1}} \\
	&= \norm{X_{2}^{T}(X_{2} - QX_{1}) + (X_{2}^{T} - (QX_{1})^{T})QX_{1}} \\
	&\leq \left(\normop{X_{2}} + \normop{X_{1}}\right)\norm{X_{2} - QX_{1}}\\
	&\leq \dfrac{9}{4} \normop{X_{1}}\norm{X_{2} - QX_{1}}.
	\end{align*}
	Taking the infimum of the right-hand side over $\mathcal{O}(d)$ gives the result.
\end{proof}

\begin{proof}[Proof of Theorem~\ref{thm::minimax}]
We restrict ourself to equivalence classes $[X]$ such that $XX^T = I_d$. There is a one-to-one mapping between such equivalence classes and elements of  $\Gkd$, since $X^T X$ is then the orthogonal projector to the space spanned by the columns of $X^T$ (see~\citep[ Section 3.4.4]{AMS08} for instance).
We get that:
\begin{align*}
\inf_{\hat{X}}\sup_{X \in \mathcal{X}}\mathbb{E}[\rho^{2}(X,\hat{X})] 
&\geq \inf_{\hat{X}}\sup_{X \in \Gkd}\mathbb{E}[\rho^{2}(X,\hat{X})] \\
\text{  (by Lemma \ref{lem::tech1})} &\geq  \left(\dfrac{4}{9\normop{X}}\right)^{2}\inf_{\hat{X}}\sup_{X \in \Gkd}\mathbb{E}[\tilde{\rho}^{2}(X,\hat{X})] \\
 & \geq \dfrac{16}{81}\inf_{\hat{X}}\sup_{X \in \Gkd}\mathbb{E}[\tilde{\rho}^{2}(X,\hat{X})], 
\end{align*}
where $\tilde{\rho}$ is the metric defined in Lemma \ref{MetricEntropy}.
Lemma \ref{lem::localPacking} allows us to consider a point $U^{*}$ and $m$ points  $\{U_{1},\ldots,U_{m}\} \subset B(U^{*},\varepsilon)$, such that 
$\alpha\varepsilon \leq \tilde{\rho}(U_{i},U_{j}) \leq 2\varepsilon$ for all $i \neq j$
and $m \geq \left(\dfrac{c_{0}}{\alpha c_{1}}\right)^{d(k - d)}$.
Without loss of generality, we can assume that $\varepsilon \leq \dfrac{1}{2}$.
By standard arguments (see for instance chapter 4 of~\citep{RigolletNotes}) we can 
lower bound the minimax risk by the probability of error for testing a finite number of hypothesis:
\begin{align*}
\inf_{\hat{X}}\sup_{X \in \Gkd}\mathbb{E}[\tilde{\rho}^{2}(X,\hat{X})] \geq (\alpha\varepsilon)^2\inf_{\psi}\max_{1\leq j \leq m}\mathbb{P}_{U_{j}}(\psi \neq j).
\end{align*}
Moreover, by Fano's inequality~\citep{FanoNotes} (see Theorem 4.19 in~\citep{RigolletNotes}) we can lower bound the probability of error for the multiple hypothesis testing problem 
$\{\mathbb{P}_{U_{i}} : i \in [m] \}$ by
\begin{align*}
\inf_{\psi}\max_{1\leq j \leq m}\mathbb{P}_{U_{j}}(\psi \neq j) \geq 1 - \dfrac{\underset{i \neq j}{\mathrm{min }} \hspace{0.1cm}\KL{\mathbb{P}_{U_{i}} || \mathbb{P}_{U_{j}}} + \log(2)}{\log(m)}. 
\end{align*}
By Lemma~\ref{LemmaKLBound},  for all pair $i,j$ we have that:
\begin{align*}
\KL{\mathbb{P}_{U_{i}}||\mathbb{P}_{U_{j}}} &\leq C\sigma^{-4}\rho^{2}(U_{i},U_{j}),
\end{align*}
where by hypothesis $\tilde{\rho}(U_{i},U_{j}) \leq 2\varepsilon$. Theorem~\ref{thm::cholesky} implies that for small enough $\varepsilon$, $\rho(U_{i},U_{j}) \leq \sqrt{2}\varepsilon$. Hence for small enough $\varepsilon$ and large enough $\sigma$ we have that
\begin{align*}
\KL{\mathbb{P}_{U_{i}}||\mathbb{P}_{U_{j}}} &\leq C\sigma^{-4}\rho^{2}(U_{i},U_{j})\\
&\leq  C\sigma^{-4}\left(\dfrac{1}{\sqrt{2}\sigma_{d}(X)}\right)^{2}\tilde{\rho}^{2}(U_{i},U_{j})\\
&\leq  C^{'}\sigma^{-4}\varepsilon^{2}. 
\end{align*}
Hence, the minimax error admits the following lower bound:
\begin{align*}
\inf_{\hat{X}}\sup_{X \in \mathcal{X}}\mathbb{E}[\rho^{2}(X,\hat{X})] 
&\geq  \dfrac{16}{81}\alpha^{2}\varepsilon^{2}\left( 1 - \dfrac{2 N C^{'}\sigma^{-4} \varepsilon^{2} + \log(2) }{d(k - d)\log(\frac{c_{0}}{\alpha c_{1}})} \right),
\end{align*}
for any $\varepsilon \in (0,\varepsilon_{0}]$ and $\alpha \in (0,1)$.
By picking
\begin{align*}
	\alpha & = \dfrac{c_{0}}{4c_{1}}, & & \textrm{and } & \varepsilon^{2} & = \dfrac{\sigma^{4}}{N}\cdot\dfrac{d(k - d)\log(2)}{6 C^{'}}
\end{align*}
if $\varepsilon_{0}^{2} \geq \dfrac{\sigma^{4}}{N}\cdot\dfrac{d(k - d)\log(2)}{6 C^{'}}$, and $\varepsilon^{2} = \dfrac{\varepsilon_{0}^{2}}{2}$ otherwise,
we get the following inequality:
\begin{align*}
\inf_{\hat{X}}\sup_{X \in \mathcal{X}}\mathbb{E}[\rho^{2}(X,\hat{X})] 
\geq  \left(\dfrac{4c_{0}}{9c_{1}}\right)^{2}\min\left( \dfrac{\sigma^{4}}{N}\dfrac{d(k - d)\log(2)}{576  C^{'}} ,\dfrac{\varepsilon_{0}^{2}}{96} \right),
\end{align*}
which yields the lower bound.
The upper bound directly follows from the upper bound on the MSE of the estimator, which can be found in Appendix~\ref{EstimatorMSE} .
\end{proof}

%% file: KLBound.tex
\section{Bound on the KL divergence}\label{KLBound}

In this section, we prove the tight bound on the $\KLop$ divergence of the distribution of samples \eqref{eq::description} stated in Lemma~\ref{LemmaKLBound}.
Our result is a direct consequence of Proposition 7.8 in~\citep{BandeiraRigolletWeed2018RatesMRA}.
\begin{proof}
By the properties of the orthogonal group and the Haar measure, we have:
\begin{align*}\mathbb{E}_{Q}[QX_{1}] &= \mathbb{E}_{Q}[QX_{2}] = 0.
\end{align*}
We now bound KL divergence between $\mathbb{P}_{X_{1}}$ and $\mathbb{P}_{X_{2}}$ by bounding the $\chi^{2}$-divergence between $\mathbb{P}_{X_{1}}$ and $\mathbb{P}_{X_{2}}$.
The density of $\mathbb{P}_{X}$, $f_{X}$ can be written as:
\begin{align*}
f_{X}(Y) = \mathbb{E}_{Q}[\sigma^{-dk}f(\sigma^{-1}(Y - QX))] = \sigma^{-dk}f(\sigma^{-1}Y)\mathbb{E}_{Q}[e^{-\frac{1}{2\sigma^{2}}(\norm{X}^{2} 
- 2\langle Y, QX \rangle)}],
\end{align*}
where $f$ is the density of a standard $dk$-dimensional Gaussian.
Since $\norm{X}^{2} \leq d$ by hypothesis, we obtain by Jensen's inequality:
\begin{equation}\label{eq::LBdensity}
f_{X}(Y) \geq \sigma^{-dk}f(\sigma^{-1}Y)e^{-\frac{1}{2\sigma^{2}}(d - 2 \mathbb{E}_{Q}[\langle Y, QX \rangle])} = \sigma^{-dk}f(\sigma^{-1}Y)e^{-\frac{d}{2\sigma^{2}}}.
\end{equation}
Hence the $\chi^{2}$ divergence can then be bounded as:
\begin{align*}
\chi^{2}(\mathbb{P}_{X_{1}},\mathbb{P}_{X_{2}}) 
&= \int \dfrac{(f_{X_{1}}(Y) - f_{X_{2}}(Y))^{2}}{f_{X_{1}}(Y)}dY\\
&\leq e^{\frac{d}{2\sigma^{2}}} \int \biggl(\mathbb{E}_{Q_{1}}\biggl[e^{-\frac{\norm{X_{1}}^{2} - 2\dotprod{Y,Q_{1}X_{1}}}{2\sigma^{2}}}\biggl] \\
 &\qquad -  \mathbb{E}_{Q_{2}}\biggl[e^{-\frac{\norm{X_{2}}^{2} - 2\dotprod{Y,Q_{2}X_{2}}}{2\sigma^{2}}}\biggl] \biggl)^{2} 
  \sigma^{-dk}f(\sigma^{-1}Y)dY\\
&= e^{\frac{d}{2\sigma^{2}}}\int \biggl(\mathbb{E}_{Q_{1}}\biggl[e^{ \dotprod{Y,\sigma^{-1}Q_{1}X_{1}} - \frac{1}{2}\norm{\sigma^{-1}X_{1}}^{2}}\biggl] \\
&\qquad - \mathbb{E}_{Q_{2}}\biggl[e^{ \dotprod{Y,\sigma^{-1}Q_{2}X_{2}} - \frac{1}{2}\norm{\sigma^{-1}X_{2}}^{2}}\biggl]\biggl)^{2}f(Y)dY \\
&=  e^{\frac{d}{2\sigma^{2}}} \mathbb{E}_{E}\biggl[\biggl(\mathbb{E}_{Q_{1}}\biggl[e^{\dotprod{E,\sigma^{-1}Q_{1}X_{1}} 
- \frac{1}{2}\norm{\sigma^{-1}X_{1}}^{2}}\biggl] \\
&\qquad -\mathbb{E}_{Q_{2}}\biggl[e^{ \dotprod{E,\sigma^{-1}Q_{2}X_{2}}- \frac{1}{2}\norm{\sigma^{-1}X_{2}}^{2}}\biggl]  \biggl)^{2}\biggl],
\end{align*}
where $E$ is a $d \times k$ matrix with i.i.d.\ entries $\sim \mathcal{N}(0,1)$.
The first inequality comes from applying inequality \eqref{eq::LBdensity} to the denominator, and the second equality comes from a change of variables.
By Fubini's theorem we can change the order of expectations. 
Using the fact that for a standard Gaussian $U$, for any $\lambda$, we have $\mathbb{E}[e^{\lambda U}] = e^{\frac{\lambda^{2}}{2}}$, we obtain that:
\begin{align*}
\mathbb{E}_{Q_{1},Q_{2}}\mathbb{E}_{E}\biggl[e^{\dotprod{E,\sigma^{-1}(Q_{1}X_{1} + Q_{2}X_{2})} - \frac{1}{2}(\norm{\sigma^{-1}X_{1}}^{2} + \norm{\sigma^{-1}X_{2}}^{2})}\biggl] 
&= \mathbb{E}_{Q_{1},Q_{2}}\biggl[e^{\frac{\dotprod{Q_{1}X_{1},Q_{2}X_{2}}}{\sigma^{2}}}\biggl].
\end{align*}
We use this to expand the square in the above expression for $\chi^{2}(\mathbb{P}_{X_{1}},\mathbb{P}_{X_{2}})$:
\begin{align*}
\chi^{2}(\mathbb{P}_{X_{1}},\mathbb{P}_{X_{2}}) 
&\leq e^{\frac{d}{2\sigma^{2}}}\biggl(\mathbb{E}_{Q_{1},\tilde{Q}_{1}}
\biggl[e^{\frac{1}{\sigma^{2}}\dotprod{Q_{1}X_{1},\tilde{Q}_{1}X_{1}}}\biggl] - 2\mathbb{E}_{Q_{1},Q_{2}}\biggl[e^{\frac{1}{\sigma^{2}}\dotprod{Q_{1}X_{1},Q_{2}X_{2}}}\biggl]  \\
&\qquad + \mathbb{E}_{Q_{2},\tilde{Q}_{2}}\biggl[e^{\frac{1}{\sigma^{2}}\dotprod{Q_{2}X_{2},\tilde{Q}_{2}X_{2}}}\biggl] \biggl)\\
&= e^{\frac{d}{2\sigma^{2}}}\biggl(\mathbb{E}_{Q_{1},\tilde{Q}_{1}}
\biggl[e^{\frac{1}{\sigma^{2}}\dotprod{X_{1},Q_{1}^{T}\tilde{Q}_{1}X_{1}}}\biggl] - 2\mathbb{E}_{Q_{1},Q_{2}}\biggl[e^{\frac{1}{\sigma^{2}}\dotprod{X_{1},Q_{1}^{T}Q_{2}X_{2}}}\biggl]  \\
&\qquad + \mathbb{E}_{Q_{2},\tilde{Q}_{2}}\biggl[e^{\frac{1}{\sigma^{2}}\dotprod{X_{2},Q_{2}^{T}\tilde{Q}_{2}X_{2}}}\biggl] \biggl)\\
&= e^{\frac{d}{2\sigma^{2}}}\mathbb{E}_{Q}\biggl[e^{\frac{1}{\sigma^{2}}\dotprod{X_{1},QX_{1}}} - 2 e^{\frac{1}{\sigma^{2}}\dotprod{X_{1},QX_{2}}} + e^{\frac{1}{\sigma^{2}}\dotprod{X_{2},QX_{2}}}\biggl].
\end{align*}
Since $Q_{1}$,$\tilde{Q}_{1}$, $Q_{2}$ and $\tilde{Q}_{2}$ are drawn according to the Haar measure on the orthogonal group,
according to Fubini's theorem we can expand each term as a power series and exchange summation and expectation. 
Using the fact that $\dotprod{u,v}^{l} = \dotprod{u^{\otimes l}, v^{\otimes l}}$ and that for all i.i.d. vectors $x$ and $y$ we have that
\begin{align*}
\mathbb{E}_{x,y}[\dotprod{x,y}^{l}] &= \mathbb{E}_{x,y}[\dotprod{x^{\otimes l},y^{\otimes l}}] 
= \dotprod{\mathbb{E}_{x}[x^{\otimes l}],\mathbb{E}_{y}[y^{\otimes l}]} 
= \norm{\mathbb{E}_{x}[x^{\otimes l}]}^{2},
\end{align*}
we obtain :
\allowdisplaybreaks
\begin{align*}
\chi^{2}(\mathbb{P}_{X_{1}},\mathbb{P}_{X_{2}}) 
&\leq e^{\frac{d}{2\sigma^{2}}}\sum_{l=0}^{\infty}\dfrac{\sigma^{-2l}}{l!}
\biggl(\mathbb{E}_{Q}[\dotprod{X_{1},QX_{1}}^{l} - 2\dotprod{X_{1},QX_{2}}^{l} + \dotprod{X_{2},QX_{2}}^{l}] \biggl)\\
&= e^{\frac{d}{2\sigma^{2}}}\sum_{l=0}^{\infty}\dfrac{\sigma^{-2l}}{l!}
\biggl(\mathbb{E}_{Q}[\dotprod{\vect{X_{1}},\vect{QX_{1}}}^{l} - 2\dotprod{\vect{X_{1}},\vect{QX_{2}}}^{l} \\
&\qquad + \dotprod{\vect{X_{2}},\vect{QX_{2}}}^{l}] \biggl)\\
&= e^{\frac{d}{2\sigma^{2}}}\sum_{l=0}^{\infty}\dfrac{\sigma^{-2l}}{l!}
\biggl(\mathbb{E}_{Q,\tilde{Q}}[\dotprod{\vect{\tilde{Q}X_{1}},\vect{QX_{1}}}^{l} - 2\dotprod{\vect{\tilde{Q}X_{1}},\vect{QX_{2}}}^{l} \\
&\qquad + \dotprod{\vect{\tilde{Q}X_{2}},\vect{QX_{2}}}^{l}] \biggl)\\
&= e^{\frac{d}{2\sigma^{2}}}\sum_{l=0}^{\infty}\dfrac{\sigma^{-2l}}{l!}
\biggl(\mathbb{E}_{Q,\tilde{Q}}[\dotprod{\vect{\tilde{Q}X_{1}}^{\otimes l},\vect{QX_{1}}^{\otimes l}} \\
&\qquad - 2\dotprod{\vect{\tilde{Q}X_{1}}^{\otimes l},\vect{QX_{2}}^{\otimes l}} + \dotprod{\vect{\tilde{Q}X_{2}}^{\otimes l},\vect{QX_{2}}^{\otimes l}} ]\biggl)\\
&= e^{\frac{d}{2\sigma^{2}}}\sum_{l=0}^{\infty}\dfrac{\sigma^{-2l}}{l!}
\biggl(\dotprod{\mathbb{E}_{\tilde{Q}}[\vect{\tilde{Q}X_{1}}^{\otimes l}],\mathbb{E}_{Q}[\vect{QX_{1}}^{\otimes l}]} \\
&\qquad - 2\dotprod{\mathbb{E}_{\tilde{Q}}[\vect{\tilde{Q}X_{1}}^{\otimes l}],\mathbb{E}_{Q}[\vect{QX_{2}}^{\otimes l}]}  \\
&\qquad+ \dotprod{\mathbb{E}_{\tilde{Q}}[\vect{\tilde{Q}X_{2}}^{\otimes l}],\mathbb{E}_{Q}[\vect{QX_{2}}^{\otimes l}]}\biggl) \\
&=e^{\frac{d}{2\sigma^{2}}}\sum_{l=0}^{\infty}\dfrac{\sigma^{-2l}}{l!} \biggl(\norm{\mathbb{E}_{\tilde{Q}}[\vect{\tilde{Q}X_{1}}^{\otimes l}]}^{2}\\
&\qquad -2\dotprod{\mathbb{E}_{\tilde{Q}}[\vect{\tilde{Q}X_{1}}^{\otimes l}],\mathbb{E}_{\tilde{Q}}[\vect{\tilde{Q}X_{2}}^{\otimes l}]}\\ 
&\qquad + \norm{\mathbb{E}_{\tilde{Q}}[\vect{\tilde{Q}X_{2}}^{\otimes l}]}^{2}\biggl)\\
&= e^{\frac{d}{2\sigma^{2}}}\sum_{l=0}^{\infty}\dfrac{\sigma^{-2l}}{l!} 
\norm{\mathbb{E}_{\tilde{Q}}[\vect{\tilde{Q}X_{1}}^{\otimes l}] - \mathbb{E}_{\tilde{Q}}[\vect{\tilde{Q}X_{2}}^{\otimes l}]}^{2}\\
&= e^{\frac{d}{2\sigma^{2}}}\sum_{l=0}^{\infty}\dfrac{\sigma^{-2l}}{l!} \norm{\Delta_{l}}^{2} \\
&= e^{\frac{d}{2\sigma^{2}}}\biggl(\sigma^{-2}\norm{\Delta_{1}}^{2} + \dfrac{\sigma^{-4}}{2!}\norm{\Delta_{2}}^{2} + \sum_{l=3}^{\infty}\dfrac{\sigma^{-2l}}{l!} \norm{\Delta_{l}}^{2}\biggl), 
\end{align*}
with $\Delta_{l} = \mathbb{E}_{Q}[\vect{QX_{1}}^{\otimes l} - \vect{QX_{2}}^{\otimes l}]$, where the expectation is taken over the Haar measure.
In particular,
\begin{align*}
\Delta_{1} &= \mathbb{E}_{Q}[\vect{QX_{1}} - \vect{QX_{2}}]
= \vect{\mathbb{E}_{Q}[Q(X_{1} - X_{2})]}
= 0.
\end{align*}
Moreover, by Lemma~\ref{ref::intermediate} below, for all $l \geq 2$ we have:
\begin{align*}
\norm{\Delta_{l}}^{2} &\leq 12(2d)^{l}\rho^{2}(X_{1},X_{2}).
\end{align*}
Hence, for any $\sigma > 1$ we obtain:
\begin{align*}
\chi^{2}(\mathbb{P}_{X_{1}},\mathbb{P}_{X_{2}}) 
&\leq e^{\frac{d}{2\sigma^{2}}}\biggl( \dfrac{\sigma^{-4}}{2}\norm{\Delta_{2}}^{2} + \sum_{l=3}^{\infty}\dfrac{\sigma^{-2l}}{l!} 12\cdot (2d)^{l}\rho^{2}(X_{1},X_{2})\biggl)\\
&\leq e^{\frac{d}{2\sigma^{2}}}\biggl( c_{1}\rho^{2}(X_{1},X_{2}) \sigma^{-4} + c_{2}\rho^{2}(X_{1},X_{2}) \sigma^{-4}\biggl)\\
&\leq C \rho^{2}(X_{1},X_{2}) \sigma^{-4}.
\end{align*}
The inequality $\KL{\mathbb{P}_{X_{1}}||\mathbb{P}_{X_{2}}} \leq \chi^{2}(\mathbb{P}_{X_{1}},\mathbb{P}_{X_{2}})$ (see Lemma 2.2 in~\citep{Tsybakov2009book} for instance) finishes the proof of the claim.
\end{proof}

The general structure of the proof of the lemma below is as in Lemma~B.12 of~\citep{BandeiraRigolletWeed2018RatesMRA}, with adaptations as needed.
 \begin{lemma}\label{ref::intermediate}
Given $X_{1}, X_{2} \in \R^{d \times k}$ such that $\rho(X_{1},X_{2}) < \dfrac{\norm{X_{1}}}{3}$, for all $l \geq 1$,
\begin{align*}
\norm{\Delta_{l}}^{2} &=  \norm{\mathbb{E}_{Q}[ \vect{QX_{1}}^{\otimes l} - \vect{QX_{2}}^{\otimes l} ]}^{2} \leq 12(2d)^{l}\rho^{2}(X_{1},X_{2}).
\end{align*}
\end{lemma}
\begin{proof}
Without loss of generality we assume that $X_{1}$ and $X_{2}$ are rotationally aligned, i.e., $\rho(X_{1},X_{2}) = \norm{X_{1} - X_{2}} =  \varepsilon\norm{X_{1}}$, with $\varepsilon < \dfrac{1}{3}$.
By Jensen's inequality,
\begin{align*}
\norm{\mathbb{E}_{Q}[ \vect{QX_{1}}^{\otimes l} - \vect{QX_{2}}^{\otimes l} ]}^{2} &\leq \mathbb{E}_{Q}[\norm{ \vect{QX_{1}}^{\otimes l} - \vect{QX_{2}}^{\otimes l} }^{2}] \\
&\leq \norm{ \vect{X_{1}}^{\otimes l} - \vect{X_{2}}^{\otimes l} }^{2}. 
\end{align*}
Expanding the norm, it is easy to check that
\begin{align*}
\norm{ \vect{X_{1}}^{\otimes l} - \vect{X_{2}}^{\otimes l} }^{2} &= \norm{\vect{X_{1}} }^{2l} - 2\dotprod{X_{1},X_{2}}^{l} + \norm{\vect{X_{2}} }^{2l}\\
&= \norm{X_{1}}^{2l}\left( 1 - 2(1 + \gamma)^{l} + (1 + 2\gamma + \varepsilon^{2})^{l}  \right),
\end{align*}
where  $\gamma = \dfrac{\dotprod{X_{1},X_{2} - X_{1}}}{\norm{X_{1}}^{2}}$. By Cauchy--Schwarz we have $|\gamma| \leq \varepsilon < \dfrac{1}{3}$.
Moreover, $2\gamma + \varepsilon^{2} \leq 3\varepsilon < 1$.
By the binomial theorem, for all $x$ such that $|x| < 1$, there exists an $r_{l} \leq 2^l x^2$ such that
\begin{align*}
(1 + x)^l = \sum_{k=0}^{l}\binom{l}{k}x^{k} = 1 + lx + r_{l}.
\end{align*}
Hence,
\begin{align*}
1 - 2(1 + \gamma)^{l} + (1 + 2\gamma + \varepsilon^{2})^{l} 
&\leq 1 - 2 -2l\gamma + 2^{l+1}\varepsilon^2 + 1 + 2l\gamma + l\varepsilon^2 + 2^l\cdot9\varepsilon^2\\
&\leq (l + 11\cdot2^l)\varepsilon^2 \leq 12\cdot 2^{l}\varepsilon^2,
\end{align*}
and $\norm{\Delta_{l}}^{2} \leq \norm{X_{1}}^{2l}12(2)^{l}\varepsilon^2 = \norm{X_{1}}^{2l-2}12(2)^{l} (\norm{X_{1}}\varepsilon)^2\leq 12(2d)^{l}\rho^{2}(X_{1},X_{2}).$ 
\end{proof}

%

%% file: MSE.tex
\section{Statistical properties of the Estimator}\label{appendix::MSE}

We now compute an upper bound on the MSE of our estimator
in order to prove the second part of Proposition~\ref{prop::regimes}.
We start by computing the covariance matrix of the estimator in part~\ref{sub::variance}, 
before proving the upper bound in part~\ref{EstimatorMSE}.

\subsection{Covariance matrix of the estimator $\hat{M}_{N}$}\label{sub::variance}

It is straightforward to get that
\begin{align*}
\mathbb{E}[\hat{M}_{N}] &= X^{T}X + \sigma^{2}dI_{k}.
\end{align*}
We now compute the covariance matrix $\Sigma$ of $\hat{M}_{N}$ i.e.
\begin{align}\label{eq::sigma}
 \Sigma &= \mathbb{E}[\vect{\hat{M}_{N} - \mathbb{E}[\hat{M}_{N} ]}\vect{\hat{M}_{N} - \mathbb{E}[\hat{M}_{N} ]}^{T}] \\
& =  \mathbb{E}[\vect{\hat{M}_{N}}\vect{\hat{M}_{N}}^{T}] - \vect{\mathbb{E}[\hat{M}_{N}]}\vect{\mathbb{E}[\hat{M}_{N}]}^{T}.
\end{align}
For that purpose, we define $F_{N}$ as
\begin{align*}
F_{N} &= \hat{M}_{N} - \mathbb{E}[\hat{M}_{N}] \\
&= \sigma\left(X^{T}\left(\dfrac{1}{N}\sum_{i=1}^{N}E_{i}\right) + \left(\dfrac{1}{N}\sum_{i=1}^{N}E_{i}\right)^{T}X \right)
+\sigma^{2}\left( \dfrac{1}{N}\sum_{i=1}^{N}E_{i}^{T}E_{i} - dI_{k}\right).
\end{align*}
Each of the entry of $F_{N}$ can be written as
\begin{align*}
(F_{N})_{s,t} &= \sigma\left(  \langle x_{s},h_{t} \rangle + \langle x_{t},h_{s}  \rangle  \right) + \sigma^{2}\left( \dfrac{1}{N}\sum_{i=1}^{N}\langle e_{s}^{(i)},e_{t}^{(i)}\rangle- d \delta_{s,t}\right), 
 \end{align*}
 where $X =\left[
  \begin{array}{cccc}
    \vrule & \vrule & & \vrule\\
    x_{1} & x_{2} & \ldots & x_{k} \\
    \vrule & \vrule & & \vrule 
  \end{array}
\right]$, $E_{i} = \left[
  \begin{array}{cccc}
    \vrule & \vrule & & \vrule\\
    e_{1}^{(i)} & e_{2}^{(i)} & \ldots & e_{k}^{(i)} \\
    \vrule & \vrule & & \vrule 
  \end{array}
\right]$ 
and \\$H_{N} = \dfrac{1}{N}\displaystyle\sum_{i=1}^{N}E_{i} = \left[
  \begin{array}{cccc}
    \vrule & \vrule & & \vrule\\
    h_{1} & h_{2} & \ldots & h_{k} \\
    \vrule & \vrule & & \vrule 
  \end{array}
\right]$.
Since the entries of $F_{N}$ have zero mean, to compute their variance we start by computing
\begin{align*}
(F_{N})_{s,t}(F_{N})_{u,v} &= \sigma^{2}\left( \langle x_{s},h_{t} \rangle\langle x_{u},h_{v} \rangle  + \langle x_{t},h_{s} \rangle \langle x_{u},h_{v} \rangle
+ \langle x_{s},h_{t} \rangle  \langle x_{v},h_{u} \rangle + \langle x_{t},h_{s} \rangle \langle x_{v}, h_{u} \rangle \right)\\
&\qquad+ \sigma^{3}\left( (\langle x_{s},h_{t} \rangle + \langle x_{t},h_{s} \rangle(\dfrac{1}{N}\sum_{i=1}^{N}\langle e_{u}^{(i)},e_{v}^{(i)}\rangle - d\delta_{u,v})) \right) \\
&\qquad+ \sigma^{3}\left( (\langle x_{u},h_{v} \rangle + \langle x_{v},h_{u} \rangle(\dfrac{1}{N}\sum_{i=1}^{N}\langle e_{s}^{(i)},e_{t}^{(i)}\rangle - d\delta_{s,t})) \right) \\
&\qquad+ \sigma^{4}\left( \dfrac{1}{N^2}\sum_{i,j} \langle e_{s}^{(i)},e_{t}^{(i)} \rangle\langle e_{u}^{(i)},e_{v}^{(i)} \rangle + d^2\delta_{s,t}\delta_{u,v} \right)\\
&\qquad -\sigma^{4} \left( d\dfrac{1}{N}\sum_{i}[\langle e_{s}^{(i)},e_{t}^{(i)} \rangle\delta_{u,v} + \langle e_{u}^{(i)},e_{v}^{(i)} \rangle\delta_{s,t}]\right). 
\end{align*}
Hence, taking the expectation we get
\begin{align*}
\mathbb{E}[(F_{N})_{s,t}(F_{N})_{u,v} ] &= \dfrac{\sigma^{2}}{N}\left( \delta_{t,v}G_{s,u} + \delta_{s,v}G_{t,u} + \delta_{t,u}G_{s,v} + \delta_{s,u}G_{t,v} \right)\\
&\qquad + \sigma^{4}d^{2}\delta_{s,t}\delta_{u,v} - \dfrac{\sigma^{4}d^{2}}{N}\sum_{i=1}^{N}[\delta_{s,t}\delta_{u,v} + \delta_{u,v}\delta_{s,t}] \\
&\qquad + \dfrac{\sigma^{4}}{N^{2}}\sum_{i,j}\left[ (1 - \delta_{i,j})(d^{2}\delta_{s,t}\delta_{u,v}) + \delta_{i,j}\mathbb{E}[\langle e_{s}^{(i)},e_{t}^{(i)} \rangle\langle e_{u}^{(i)},e_{v}^{(i)} \rangle] \right].
\end{align*}
Equivalently,
 \begin{align*}
\mathbb{E}[(F_{N})_{s,t}(F_{N})_{u,v} ] 
&= \dfrac{\sigma^{2}}{N}\left( \delta_{t,v}G_{s,u} + \delta_{s,v}G_{t,u}  + \delta_{t,u}G_{s,v}   + \delta_{s,u}G_{t,v} \right)\\
& \qquad + \sigma^{4}d^{2}\delta_{s,t}\delta_{u,v}\left( 1 - \dfrac{2N}{N} + \dfrac{N^{2} - N}{N^{2}} \right) + \dfrac{\sigma^{4}}{N}A_{s,t,u,v},
\end{align*}
where $E = \left[
  \begin{array}{cccc}
    \vrule & \vrule & & \vrule\\
    e_{1} & e_{2} & \ldots & e_{k} \\
    \vrule & \vrule & & \vrule 
  \end{array}\right]$ is a $d \times k$ random matrix with i.i.d.\ Gaussian coefficients and $A_{s,t,u,v} = \mathbb{E}[\langle e_{s},e_{t} \rangle\langle e_{u},e_{v} \rangle]$.
By expanding $A_{s,t,u,v}$, we can write
\begin{align*}
A_{s,t,u,v} &= \mathbb{E}\left[\sum_{\alpha}e_{\alpha,s}e_{\alpha,t}\sum_{\beta}e_{\beta,u}e_{\beta,v}\right]\\
&= \sum_{\alpha,\beta}\mathbb{E}\left[e_{\alpha,s}e_{\alpha,t}e_{\beta,u}e_{\beta,v}\right]\\
&= \sum_{\alpha}\mathbb{E}\left[e_{\alpha,s}e_{\alpha,t}e_{\alpha,u}e_{\alpha,v}\right] + \sum_{\alpha \neq \beta}\mathbb{E}\left[e_{\alpha,s}e_{\alpha,t}]\mathbb{E}[e_{\beta,u}e_{\beta,v}\right]\\
&= d\cdot\mathbb{E}\left[z_{s}z_{t}z_{u}z_{v}\right] + (d^2 - d)\delta_{s,t}\delta_{u,v}.
\end{align*} 
In order to compute $\mathbb{E}\left[z_{s}z_{t}z_{u}z_{v}\right]$, we distinguish seven cases below:
\begin{align*}
s = t = u = v &: \mathbb{E}[z_{s}z_{t}z_{u}z_{v}] = 3  \\
s = t \neq u = v &: \mathbb{E}[z_{s}z_{t}z_{u}z_{v}] = 1 \\
s = t \text{ and } u \neq v &: \mathbb{E}[z_{s}z_{t}z_{u}z_{v}] = 0 \\
s = u \neq t = v &: \mathbb{E}[z_{s}z_{t}z_{u}z_{v}] = 1 \\
s = u \text{ and } t \neq v &: \mathbb{E}[z_{s}z_{t}z_{u}z_{v}] = 0 \\
s = v \neq t = u &: \mathbb{E}[z_{s}z_{t}z_{u}z_{v}] = 1 \\
s = v \text{ and } t \neq u &: \mathbb{E}[z_{s}z_{t}z_{u}z_{v}] = 0.
\end{align*}
In other words,
\begin{align*}
\mathbb{E}[z_{s}z_{t}z_{u}z_{v}] &= \delta_{s,t}\delta_{u,v} +  \delta_{s,u}\delta_{t,v} +  \delta_{s,v}\delta_{t,u} =  \delta_{s,u}\delta_{t,v} +  \delta_{s,v}\delta_{t,u}.
\end{align*}
We therefore get that the entry of matrix $\Sigma$ defined in~\eqref{eq::sigma} corresponding to the covariance between elements $(s,t)$ and $(u,v)$ of $\hat{M}_{N}$ is
\begin{align*}
\Sigma_{(s,t),(u,v)} &= \mathbb{E}[(F_{N})_{s,t}(F_{N})_{u,v} ] \\
&= \dfrac{1}{N}\left[ \sigma^{2}(\delta_{t,v}G_{s,u} + \delta_{s,v}G_{t,u} + \delta_{s,u}G_{t,v}) 
+ \sigma^{4}d(\delta_{s,u}\delta_{t,v} + \delta_{s,v}\delta_{t,u})\right],
\end{align*}
where $G = X^{T}X$.

 \subsection{Mean Squared Error of $\hat{G}_{N}$}\label{EstimatorMSE}

Using Lemma~\ref{lem:TuLemma} for example~\citep[Lem.~5.4]{Tuetal2016LowRankProcrustesFlow}, we find
\begin{align*}
\rho^{2}(X,\hat{X})\leq \dfrac{1}{2(\sqrt{2} - 1)\sigma_{d}^{2}(X)}\norm{G - \tilde{G}_{N}}^{2},
\end{align*}
where $\tilde{G}_{N} = \underset{H\succeq 0 : \mathrm{rank}(H) \leq d}{\mathrm{argmin}} \norm{ \hat{G}_N - H}$
is the projection of $\hat{G}_{N}$ on the cone of positive semidefinite matrices of rank at most $d$. 
Since $\mathrm{rank}(G) = d$,  $\tilde{G}_{N}$ satisfies
\begin{align*}
\norm{\hat{G}_{N} - \tilde{G}_{N}} &\leq \norm{\hat{G}_{N} - G}.
\end{align*}
By triangle inequality we have
\begin{align*}
\norm{G - \tilde{G}_{N}} &\leq \norm{G - \hat{G}_{N}} + \norm{\hat{G}_{N} - \tilde{G}_{N}} \leq 2 \norm{G - \hat{G}_{N}}. 
\end{align*}
Hence
\begin{align}\label{eqrefG1}
\rho^{2}(X,\hat{X})\leq \dfrac{2}{(\sqrt{2} - 1)\sigma_{d}^{2}(X)}\norm{G - \hat{G}_{N}}^{2}.
\end{align}
Building on the computations from subsection~\ref{sub::variance}, we get that
\begin{align}
\mathbb{E}[\norm{G - \hat{G}_{N}}^{2}] &= \mathbb{E}[\norm{F_{N}}^{2}] \nonumber\\
&= \sum_{s=1}^{k}\sum_{t=1}^{k}\mathbb{E}[(F_{N})_{s,t}^{2}] \nonumber\\
&= \dfrac{\sigma^{2}}{N}\left[\sum_{s}(4G_{s,s} + 2\sigma^{2}d) + \sum_{s \neq t}(G_{s,s} + G_{t,t} + \sigma^{2}d)\right]\nonumber\\
&= \dfrac{\sigma^{2}}{N}\left[k(k+1)\sigma^{2}d + 4\Tr{G} + 2(k-1)\Tr{G} \right] \nonumber\\
&= \dfrac{(k+1)\sigma^{2}}{N}\left[ k\sigma^{2}d + \norm{X}^{2} \right].
\label{eqrefG2}
\end{align}
Finally,~\eqref{eqrefG1} and~\eqref{eqrefG2} imply that
\begin{align*}
\mathbb{E}[ \rho^{2}(X,\hat{X})] \leq \dfrac{2(k+1)\sigma^{2}}{N(\sqrt{2} - 1)\sigma_{d}^{2}(X)}\left[ k\sigma^{2}d + \norm{X}^{2} \right] = O\left(\dfrac{\sigma^{2} + \sigma^{4} }{N}\right).
\end{align*}
This directly implies that, in the case where $\sigma \ll 1$, we have $\mathbb{E}[ \rho^{2}(X,\hat{X})] = O\left(\dfrac{\sigma^{2}}{N}\right)$,
 while in the case where $\sigma \gg 1$, we have $\mathbb{E}[ \rho^{2}(X,\hat{X})] = O\left(\dfrac{\sigma^{4}}{N}\right)$.